\documentclass[11pt]{article}
\usepackage{amsmath}
\usepackage[thmmarks,amsmath,amsthm]{ntheorem}
\usepackage{comment}

\usepackage{amsfonts, psfrag, graphicx, amssymb,enumitem,indentfirst,float,geometry,color,enumerate,graphicx,subcaption,algorithm,algpseudocode,pifont,hyperref,mathtools,gensymb}
\usepackage{pbox}
\usepackage{booktabs, cellspace, hhline}
\allowdisplaybreaks[4]
\numberwithin{equation}{section}
\numberwithin{figure}{section}

\usepackage{etoolbox}
\patchcmd{\thebibliography}{\chapter*}{\section*}{}{}
\usepackage[title]{appendix}

\usepackage{lipsum}

\newtheorem{thm}{Theorem}[section]
\newtheorem{lemma}{Lemma}[section]
\newtheorem{rem}{Remark}[section]

\numberwithin{table}{section}

\newcommand{\commentout}[1]{{}} 

\newcommand{\bfb}{{\bf b}}

\newcommand{\bfc}{{\bf c}}

\newcommand{\bfdelta}{\boldsymbol{\delta}}

\newcommand{\bfgamma}{\boldsymbol{\gamma}}

\newcommand{\vertiii}[1]{{\left\vert\kern-0.25ex\left\vert\kern-0.25ex\left\vert #1
    \right\vert\kern-0.25ex\right\vert\kern-0.25ex\right\vert}}

\begin{document}
\title{A Trilinear Immersed Finite Element Method \\ for Solving Elliptic Interface
Problems}
\author{
Ruchi Guo \thanks{Department of Mathematics, Virginia Tech, Blacksburg, VA 24061 (ruchi91@vt.edu) }
\and Tao Lin \thanks{Department of Mathematics, Virginia Tech, Blacksburg, VA 24061 (tlin@vt.edu) }
  }
\date{}
\maketitle
\begin{abstract}
This article presents an immersed finite element (IFE) method for solving the typical three-dimensional second order elliptic interface problem with an interface-independent Cartesian mesh. The local IFE space on each interface element consists of piecewise trilinear polynomials which are constructed by extending polynomials from one subelement to the whole element according to the jump conditions of the interface problem. In this space, the IFE shape functions with the Lagrange degrees of freedom can always be constructed regardless of interface location and discontinuous coefficients. The proposed IFE space is proven to have the optimal approximation capabilities to the functions satisfying the jump conditions. A group of numerical examples with representative interface geometries are presented to demonstrate features of the proposed IFE method.
\end{abstract}

\section{Introduction}

Let $\Omega\subseteq\mathbb{R}^3$ be a domain, and, without loss of generality, we assume that
$\Omega$ is separated into two subdomains $\Omega^-$ and $\Omega^+$ by a closed $C^2$ interface surface $\Gamma\subseteq\Omega$. These subdomains contain different materials identified by a piecewise constant parameter $\beta$ discontinuous across the interface $\Gamma$, i.e.,
\begin{equation*}
\beta(X)=
\left\{\begin{array}{cc}
\beta^- & \text{in} \; \Omega^- ,\\
\beta^+ & \text{in} \; \Omega^+.
\end{array}\right.
\end{equation*}
We consider the following interface problem of the elliptic type on $\Omega$:
\begin{subequations}\label{model}
\begin{align}
\label{inter_PDE}
& -\nabla\cdot(\beta\nabla u)=f, ~~~~~~~~~~~~~~~~~~~~~~~~~~~~~~~~~~~~~~~~~~~  \text{in} \;\; \Omega^-  \cup \Omega^+, \\
& [u]_{\Gamma}:  = u|_{\Omega^+} - u|_{\Omega^-} = 0, ~~~~~~~~~~~~~~~~~~~~~~~~~~~~~~~~~~  \text{on} \;\; \Gamma \label{jump_cond_1}, \\
& \big[\beta \nabla u\cdot \mathbf{n}\big]_{\Gamma} := \beta^+ \nabla u|_{\Omega^+}\cdot \mathbf{n} - \beta^- \nabla u|_{\Omega^-}\cdot \mathbf{n} = 0, \;\;\;\;~  \text{on} \;\; \Gamma \label{jump_cond_2}, \\
&  u=g, ~~~~~~~~~~~~~~~~~~~~~~~~~~~~~~~~~~~~~~~~~~~~~~~~~~~~~~~~~~\text{on} \;\; \partial\Omega,
\end{align}
\end{subequations}
where $\mathbf{ n}$ is the normal vector to $\Gamma$.
For simplicity, we denote $u^s=u|_{\Omega^s}$, $s=\pm$, in the rest of this article.

The elliptic interface problem \eqref{model} has wide applications in science and engineering such as inverse problems \cite{2018GuoLinLinIntInvProb,2005HolderDavid,2010VallaghePapadopoulo}, fluid dynamics \cite{1997LevequeLi,2001LiLai}, biomolecular electrostatics \cite{2002FogolariBrigoMolinari,2016XieYing}, plasma simulation \cite{1991BirdsallLangdon,1988HockneyEastwood}, to name just a few. Traditional finite element methods can be applied to solve this interface problem based on an interface-fitted mesh \cite{1970Babuska,1998ChenZou,1982Xu}. However, when the interface has complex geometry, for example the material interface in biomedical images \cite{2009FormaggiaQuarteroniVeneziani,2010VallaghePapadopoulo} and geophysical images \cite{2014DassiPerottoFormaggiaRuffo} in the 3-D case, it is a time-consuming and non-trivial process to generate a high-quality interface-fitted mesh to resolve the interface geometry. And this mesh generation issue will become more severe if the interface changes its shape or moves in computation. Recently, a new finite element method based on a semi-structured mesh was proposed in \cite{2017ChenWeiWen} where the interface geometry is fitted by a local Delaunay triangulation on interface elements of a pre-generated background Cartesian mesh.

Alternatively, methods that can solve the interface problem \eqref{model} on a mesh independent of the interface geometry, referred as the unfitted mesh methods, have drawn attention from researchers. Methods in this category can be roughly categorized into two groups: modify computation scheme around the interface or modify finite element functions on interface elements. Examples in the first group are the immersed interface methods (IIM) \cite{1994LevequeLi,2006LiIto} in the finite difference context and the CutFEM \cite{2015BurmanClaus,2002HansboHansbo} based on the finite element scheme. Methods in the second group can be found for the multiscale finite element methods \cite{2010ChuGrahamHou,2009EfendievHou}, the extended finite element methods \cite{2001DolbowMoesBelytschko,2001SukumarChoppMoesBelytschko}, the partition of unity methods \cite{1996BabuskaMelenk,2004SukumarHuang} and the immersed finite element (IFE) methods to be discussed in this article. We note that some of these methods may actually involve both the two types of modifications.

The key idea in the IFE methods is to use piecewise polynomials constructed according to the jump conditions on interface elements, i.e., the Hsieh-Clough-Tocher \cite{2001Braess,1966CloughTocher} type macro elements, to capture the jump behaviors across the interface, while standard polynomials are used over non-interface elements. In addition to the optimal convergence rate the IFE method can achieve on a unfitted mesh, it can also keep the number and location of degrees of freedom isomorphic to the standard finite element method defined on the same mesh. And this feature is advantageous when dealing with moving interface problems for which we refer readers to \cite{2018AdjeridChaabaneLinYue,2018BaiCaoHeLiuYang,2018GuoLinLinIntInvProb,2013HeLinLinZhang,2013LinLinZhang1}.

In this work, we discuss a trilinear IFE method on a highly structured mesh (Cartesian mesh) for solving the interface problem \eqref{model} with optimal accuracy. Our research presented here is motivated by real-world problems. For example, in plasma simulations in composite materials by the particle-in-cell (PIC) code \cite{1991BirdsallLangdon,1988HockneyEastwood}, many macro-particles modeling a plasma in the self-consistent electromagnetic field have to be traced and located on elements in a mesh iteratively during their motion in order to perform the particle-mesh interpolation procedure; therefore a structure mesh is preferred to a unstructured mesh because of the computational cost in search. As an another example, in electroencephalography, a background Cartesian mesh for a head model can be constructed by the pixels of magnetic resonance images (MRIs) \cite{2010VallaghePapadopoulo} for finite element computation. We refer readers to \cite{2005KafafyLinLinWang,2010VallaghePapadopoulo,2008WangHeCao} for the applications of IFE methods in these fields.

Although many works on the 2-D IFE methods have appeared in the literature, such as \cite{2019Guo,2016GuoLin,2018GuoLinLinElasAprox,2016GuoLinZhang,2018GuoLinZhuang,2008HeLinLin,2015LinLinZhang,2015LinYangZhang1} for theoretical analysis and \cite{2015AdjeridChaabaneLin,2018AdjeridChaabaneLinYue,2018BaiCaoHeLiuYang,2018GuoLinLinIntInvProb,2008WangHeCao}
for applications, to name just a few, the study on the 3-D case is relatively sparse, see \cite{2005KafafyLinLinWang} for a linear IFE method and \cite{2010VallaghePapadopoulo} for a trilinear IFE method. Even though the research reported here is within the direction of that in \cite{2010VallaghePapadopoulo}, but our work has three distinct new contributions. The first one is a group of detailed geometric estimates for a suitable linear approximation of the interface surface on each interface element. In particular, we introduce a maximal angle condition for constructing such a special linear approximation to the interface surface with optimal accuracy in terms of the surface curvature and mesh size. Both the theoretical analysis and numerical experiments indicate that this maximal angle condition and the resulted geometric properties are the foundation of the optimal approximation capabilities for the proposed IFE spaces. We also believe these fundamental geometric estimates can be useful for other unfitted mesh methods. Secondly, on interface elements, we construct local IFE spaces by extending trilinear polynomials from one subelement to another through a discretized extension operator designed according to the jump conditions across the interface. IFE shape functions with some desirable features, such as the Lagrange type degrees of freedom, can be readily constructed from this space and their existence is guaranteed completely independent of the interface location and coefficients $\beta^{\pm}$. Moreover, the extension operator enables us to establish the optimal approximation capabilities for the proposed IFE spaces. To the best of our knowledge, this is the first work for 3-D IFE methods covering both the development and basic analysis. We highlight that the proposed construction and analysis techniques can be readily extendable to IFE methods on unstructured unfitted meshes for solving the elliptic interface problems.

This article consists of four additional sections. In the next section, we establish a group of fundamental geometric estimates. In Section \ref{sec:tri_IFE_func}, we develop the trilinear IFE space and construct the Lagrange IFE shape functions. In Section \ref{sec:approximation}, we analyze the approximation capabilities of the proposed IFE space. In the last section, we present a group of numerical examples with representative geometries of the interface surface to demonstrate the features of the proposed IFE method.





\section{Some Geometries of Interface Elements}
\label{sec:geometry}

In this section, we present some geometric properties related to the interface surface and interface elements. These properties are fundamental for both the construction of the IFE spaces to be proposed and the related error analysis.

Throughout this article, we assume the bounded domain $\Omega\subset\mathbb{R}^3$ is a union of finitely many rectangular parallelepipeds,
and let $\mathcal{T}_h$ be a Cartesian mesh of the domain $\Omega$ with the maximum length of edge $h$. We let $\mathcal{F}_h$ and $\mathcal{E}_h$ be the collection of faces and edges in this mesh $\mathcal{T}_h$. We call an element $T\in \mathcal{T}_h$ an interface element if $T\cap \Gamma \neq \emptyset$; otherwise, we call it a non-interface element. Similarly, we can define the interface faces and edges. Furthermore, we use $\mathcal{T}^i_h$/$\mathcal{F}^i_h$/$\mathcal{E}^i_h$ and $\mathcal{T}^n_h$/$\mathcal{F}^n_h$/$\mathcal{E}^n_h$ to denote the collection of interface and non-interface elements/faces/edges, respectively.

Given each measurable subset $\tilde \Omega \subseteq \Omega$, let $W^{k,p}(\tilde \Omega)$ be the standard Sobolev spaces on $\tilde \Omega$ with the Sobolev norm $\|\cdot\|_{k,p,\tilde \Omega}$ and the semi-norm $|v|_{k,p,\tilde \Omega}=\|D^{\alpha}v\|_{0,p,\tilde \Omega}$, for $|\alpha|=k$. The corresponding Hilbert space is $H^k(\tilde \Omega)=W^{k,2}(\tilde \Omega)$ associated with the norm $\|\cdot\|_{k,\tilde \Omega}$ and semi-norm $|\cdot|_{k,\tilde \Omega}$. In the case $\tilde \Omega^s := \tilde \Omega \cap \Omega^s \not = \emptyset, s = \pm$, we define the splitting Hilbert space
\begin{equation}
\label{split_Hspa}
PH^k(\tilde{\Omega}) = \{ u\in H^k(\tilde{\Omega}^{\pm}) ~:~  [u]|_{\Gamma\cap\tilde{\Omega}}=0 ~ \text{and} ~  [\nabla u\cdot\mathbf{ n}]|_{\Gamma\cap\tilde{\Omega}}=0 \},
\end{equation}
where the definition implicitly implies the involved traces on $\Gamma\cap\tilde{\Omega}$ are well defined, with the associated norms
\begin{equation*}
\|\cdot\|^2_{k,\tilde \Omega}=\|\cdot\|^2_{k,\tilde \Omega^+}+\|\cdot\|^2_{k,\tilde \Omega^-}, \;\;\;\;\; |\cdot|^2_{k,\tilde \Omega}=|\cdot|^2_{k,\tilde \Omega^+}+|\cdot|^2_{k,\tilde \Omega^-},
\end{equation*}
\begin{equation*}
\|\cdot\|_{k,\infty,\tilde \Omega}=\max(\|\cdot\|_{k,\infty,\tilde \Omega^+} \;,\; \|\cdot\|_{k,\infty,\tilde \Omega^-}), \;\;\;\;\; |\cdot|_{k,\infty,\tilde \Omega}=\max(|\cdot|_{k,\infty,\tilde \Omega^+} \;,\; |\cdot|_{k,\infty,\tilde \Omega^-}).
\end{equation*}
Furthermore, for each interface element, we define its patch $\omega_T$ as
\begin{equation}
\label{patch}
\omega_T = \{ T'\in \mathcal{T}_h~:~ \overline{T'}\cap \overline{T} \neq \emptyset \}.
\end{equation}

 We begin by recalling the definition and existence of a so called $r$-tubular neighborhood of the smooth interface surface $\Gamma$ which are based on the following Lemma from \cite{1959Federer}.

\begin{lemma}[$r$-tubular neighborhood]
\label{tubular}
Given a smooth compact surface $\Gamma$ in $\mathbb{R}^3$, for each $X\in\Gamma$, let $N_X(r)$ be a segment with the length $2r$ centered at $X$ and perpendicular to $\Gamma$. Then, there exists a positive $r>0$ such that $N_X(r)\cap N_Y(r) = \emptyset$
for any $X, Y\in \Gamma, X \not = Y$.
\end{lemma}
The $r$-tubular neighborhood of $\Gamma$ is defined as the set $U_{\Gamma}(r)=\cup_{X\in\Gamma} N_X(r)$.
The existence of the $r$-tubular neighborhood of a smooth surface is given in \cite{1959Federer}. Define $r_{\Gamma}$ as the largest $r$ such that Lemma \ref{tubular} holds, and this positive number $r_{\Gamma}$ is referred as the reach of the surface $\Gamma$ in some literature \cite{2002MorvanThibert}. In the following discussion, we assume that

\begin{itemize}[leftmargin=30pt]
  \item[(\textbf{H1})] $h<r_{\Gamma}/(3\sqrt{3})$. 
  \item[(\textbf{H2})] The interface surface $\Gamma$ can not intersect any edge $e\in\mathcal{E}_h$ at more than one point.
  \item[(\textbf{H3})] The interface surface $\Gamma$ can not intersect the boundary of any face $f\in\mathcal{F}_h$ at more than two points.
\end{itemize}
We note that a similar assumption as (\textbf{H1}) about the $r$-tubular neighborhood has been used in a 2D unfitted mesh method \cite{2015GuzmanSanchezSarkisP1}. And the assumptions (\textbf{H2}) and (\textbf{H3}) basically mean the interface surface is resolved enough by the unfitted mesh, and these assumptions have been used in many works on unfitted meshes such as \cite{2015BurmanClaus,2016GuoLin,2002HansboHansbo,2005KafafyLinLinWang}. These assumptions can be satisfied when the interface surface is flat enough locally inside each interface element which holds in general when the mesh size is sufficiently small. In particular, for each interface element $T \in \mathcal{T}_h^i$, (\textbf{H1}) guarantees that its patch $\omega_T$ is inside the $r_\Gamma$-tubular neighborhood of the interface surface.

Based on the assumptions (\textbf{H2}) and (\textbf{H3}), an interface surface can only intersect a cubic interface element $T$ with at least three faces but no more than six faces. Therefore, by considering rotations, we classify the interface element configuration according to the number of interface faces of $T$: only one possible configuration for three interface faces as shown by Case 1 in Figure \ref{fig:subfig_cub}, two possible configurations for four interface faces as shown by Case 2 and Case3 in Figure \ref{fig:subfig_cub}, one possible configuration for five interface faces as shown by Case 4 in Figure \ref{fig:subfig_cub} and one possible configuration for six interface faces as shown by Case 5 in Figure \ref{fig:subfig_cub}. Moreover, it can be considered as certain limit situations of those interface configurations in Figure \ref{fig:subfig_cub} when some vertices of $T$ are on the interface surface; and our construction and analysis techniques developed below are readily extended to handle these situations. Therefore, in the following discussion, for the simplicity of presentations, we only discuss these five interface element configurations. This classification approach can be easily implemented to determine which configuration an interface element belongs to by counting the interface faces and location of vertices relative to the interface.

\begin{figure}[H]
\centering
\begin{subfigure}{.25\textwidth}
     \includegraphics[width=1.5in]{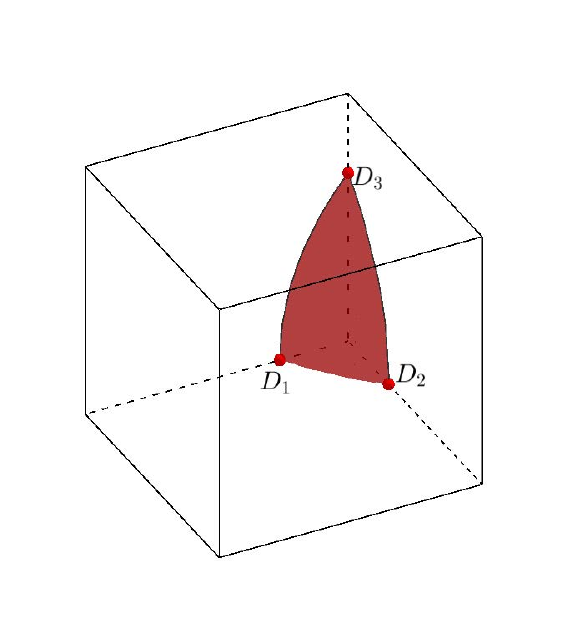}
     \label{cub_case1} 
     \caption{Case 1}
\end{subfigure}
\begin{subfigure}{.25\textwidth}
     \includegraphics[width=1.5in]{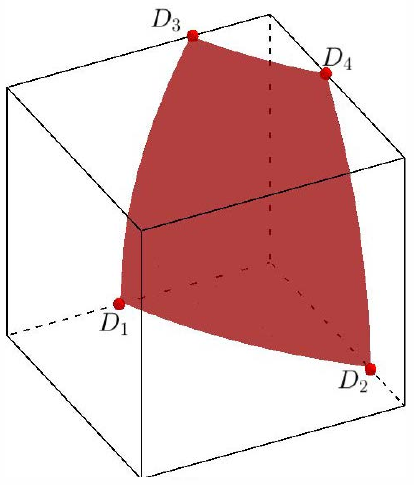}
     \label{cub_case2} 
     \caption{Case 2}
\end{subfigure}
 \begin{subfigure}{.25\textwidth}
     \includegraphics[width=1.5in]{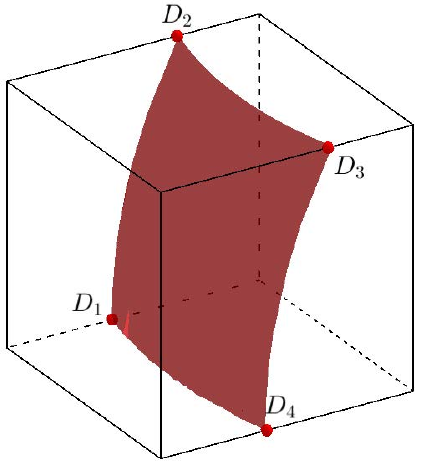}
     \label{cub_case3} 
     \caption{Case 3}
\end{subfigure}
 \begin{subfigure}{.25\textwidth}
     \includegraphics[width=1.5in]{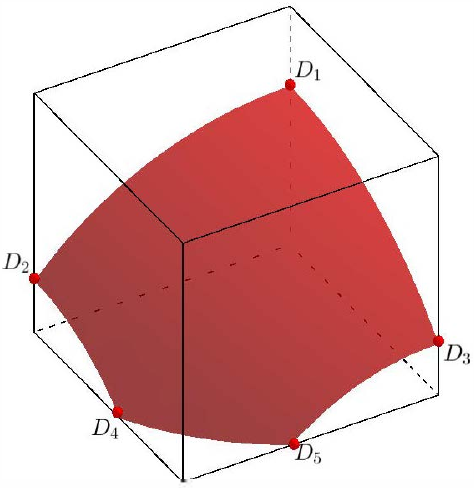}
     \label{cub_case4} 
     \caption{Case 4}
\end{subfigure}
\begin{subfigure}{.25\textwidth}
     \includegraphics[width=1.5in]{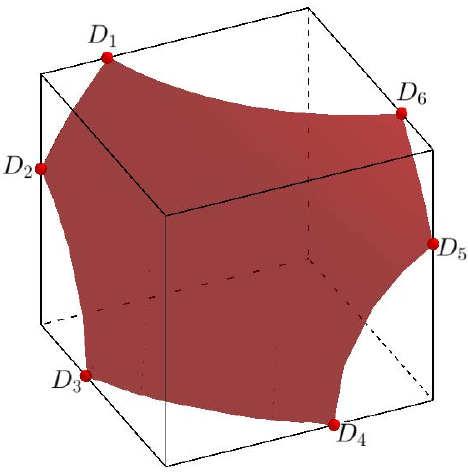}
     \label{cub_case5} 
     \caption{Case 5}
\end{subfigure}
     \caption{Possible Interface Element Configuration }
  \label{fig:subfig_cub} 
\end{figure}

The fundamental idea of the IFE methods is the employment of piecewise polynomials constructed on interface elements to satisfy
the jump conditions in a certain approximate sense. On each interface element $T$, a 2D IFE function is a piecewise polynomial defined according to the two subelements formed by the straight line connecting the intersection points of the interface and $\partial T$, and this
line is a natural approximation to the interface with sufficient accuracy, see \cite{2016GuoLin,2008HeLinLin,2004LiLinLinRogers} and the reference therein. However, we note that, in 3D, constructing a suitable linear approximation for an arbitrary interface surface is not as straightforward as the 2-D case because of at least two basic issues. The first one, already reported in \cite{2005KafafyLinLinWang}, is the issue that the intersection points of the interface and the edges of an interface element are usually not coplanar so that it is not always clear how to use these points to form a linear approximation to the interface surface inside this interface element. The second one concerns the accuracy for a plane to approximate the interface surface. We address these two issues in the following discussion.


We start from recalling some useful geometric quantities and estimates for the interface surface $\Gamma$ from \cite{2002MorvanThibert}. Denote the maximum curvature of $\Gamma$ by $\kappa$. Consider an arbitrary triangle $K = \bigtriangleup B_1B_2B_3$ with
$B_i \in \Gamma, i = 1, 2, 3$ and its normal $\bar{\mathbf{ n}}(K)$. Let $\mathcal{S}_{\Gamma}(K)$ be the subset of $\Gamma$ such that its projection onto the plane determined by $K$ is exactly $K$. To facilitate a simple presentation, we assume that the projection from $S_{\Gamma}(K)$ to $K$ is bijective. Let $\alpha_{\Gamma}(K)\in[0,\pi]$ be the maximum angle between $\bar{\mathbf{ n}}(K)$ and normal vectors of $\mathcal{S}_{\Gamma}(K)$, then we can let $\bar{\mathbf{ n}}(K)$ have the direction such that $\alpha_{\Gamma}(K)\in[0,\pi/2]$. Further define
\begin{align}
&\lambda(K) = \max\{\sin(\angle B_1B_2B_3), \sin(\angle B_2B_3B_1), \sin(\angle B_3B_1B_2) \}, \label{lambda(K)}\\
& l(K)=\max\{|\overline{B_1B_2}|, |\overline{B_2B_3}|, |\overline{B_3B_1}| \}, \label{l(K)}
\end{align}
and let $\mathcal{H}_{\Gamma}(K)$ be the Hausdorff distance between the set $K$ and $\mathcal{S}_{\Gamma}(K)$. Then, we recall Theorem 3 in \cite{2002MorvanThibert}:
\begin{thm}
\label{thm_angle_est}
Assume the projection from $\mathcal{S}_{\Gamma}(K)$ onto $K$ is bijective, $l(K)<r_{\Gamma}$ and $4\lambda(K)(1 - \kappa\mathcal{H}_{\Gamma}(K))^4 - \kappa^2l(K)^2 - 4\kappa l(K)>0$, then $\kappa\mathcal{H}_{\Gamma}(K)<1$ and
\begin{equation}
\label{angle_est_eq0}
\sin{(\alpha_{\Gamma}(K))} \leqslant \kappa l(K) \left( \frac{1}{1 - \kappa\mathcal{H}_{\Gamma}(K)} + \frac{\kappa ~ l(K) +4}{ 4\lambda(K)(1 - \kappa\mathcal{H}_{\Gamma}(K))^4 - \kappa^2l(K)^2 - 4\kappa l(K) } \right).
\end{equation}
\end{thm}
The estimate given by this theorem quantifies the flatness of the interface surface locally in terms of the angle between the normal vectors determined by $K = \bigtriangleup B_1B_2B_3$ and $\Gamma$. More importantly, assume the edges of an interface element $T\in\mathcal{T}^i_h$ intersect with
$\Gamma$ at $D_i$, $1 \leq i \leq I_T, 3 \leq I_T \leq 6$, see Figure \ref{fig:subfig_cub}, this quantification motivates us to construct a plane $\tau(T)$ to approximate $\Gamma\cap T$ as the one determined by the triangle $K_T = \bigtriangleup D_{j_1}D_{j_2}D_{j_3}$ with $D_{j_i}, i = 1, 2, 3$ chosen as follows:
\commentout{
passing three suitably chosen element-interface intersection points such that the
maximum angle of the triangle determined by these three element-interface intersection points is bounded above away from $\pi$. The analysis to be presented later shows that this maximum angle feature guarantees the plane constructed by the proposed procedure can approximate
 $T\cap \Gamma$ with sufficient accuracy. Specifically, let $D_i$, $1 \leq i \leq I_T, 3 \leq I_T \leq 6$ be the element-interface intersection points of an interface element $T\in\mathcal{T}^i_h$, see Figure \ref{fig:subfig_cub}, and we let $\tau(T)$ be the plane determined by
 $K_T = \bigtriangleup D_{j_1}D_{j_2}D_{j_3}$ with $D_{j_i}, i = 1, 2, 3$ chosen as follows:
}
\begin{itemize}[leftmargin=45pt]
\item[Case 1.] $D_{j_i} = D_i, i = 1, 2, 3$.
\item[Case 2.] $D_{j_i}, i = 1, 2, 3$ are chosen such that their distances to the non-interface edge isolated
by the interface surface $\Gamma$ from other non-interface edges of $T$ are the largest possible. For example, in Figure \ref{fig:subfig_interf}(\subref{inter_elem_case2}),
      this non-interface edge is $A_1A_5$ and we let $D_{j_1} = D_1, D_{j_2} = D_2, D_{j_3} = D_4$ because
      $D_2A_1\geqslant D_1A_1\geqslant D_4A_5\geqslant D_3A_5$.
\commentout{
  {\color{red}need to refine this part: passing through three of the intersection points with the three largest distance to the edge $A_1A_5$, e.g., the points $D_1$, $D_2$ and $D_4$ in Figure \ref{fig:subfig_interf}(\subref{inter_elem_case2}) with $D_2A_1\geqslant D_1A_1\geqslant D_4A_5\geqslant D_3A_5$.}
  }
\item[Case 3.] $D_{j_i}, i = 1, 2, 3$ are arbitrarily chosen from $D_i$, $1 \leq i \leq 4$.
\item[Case 4.] $D_{j_i}, i = 1, 2, 3$ are located on  the three parallel edges, e.g., $D_1$, $D_2$ and $D_3$ in Figure \ref{fig:subfig_interf}(\subref{inter_elem_case4}).
\item[Case 5.] $D_{j_i}, i = 1, 2, 3$ are located on  the three orthogonal edges, e.g., $D_1$, $D_3$, $D_5$ or $D_2$, $D_4$, $D_6$ in Figure \ref{fig:subfig_interf}(\subref{inter_elem_case5}).
\end{itemize}
\begin{figure}[H]
\centering
\begin{subfigure}{.25\textwidth}
     \includegraphics[width=1.5in]{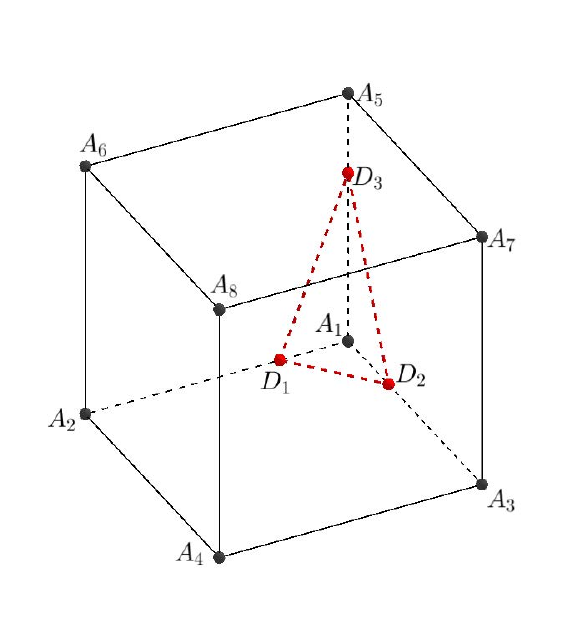}
     \caption{Case 1}
     \label{inter_elem_case1} 
\end{subfigure}
\begin{subfigure}{.25\textwidth}
     \includegraphics[width=1.7in]{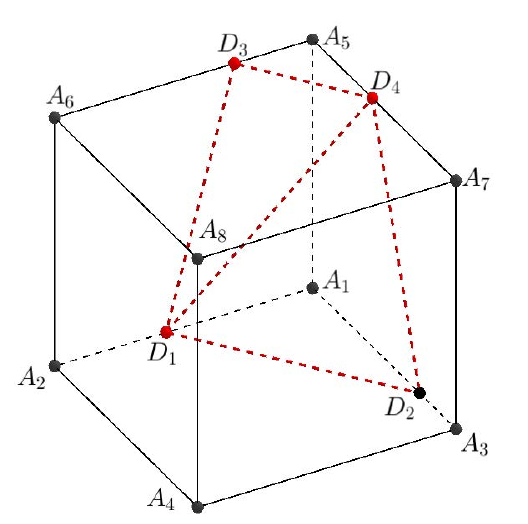}
     \caption{Case 2}
     \label{inter_elem_case2} 
\end{subfigure}
 \begin{subfigure}{.25\textwidth}
     \includegraphics[width=1.5in]{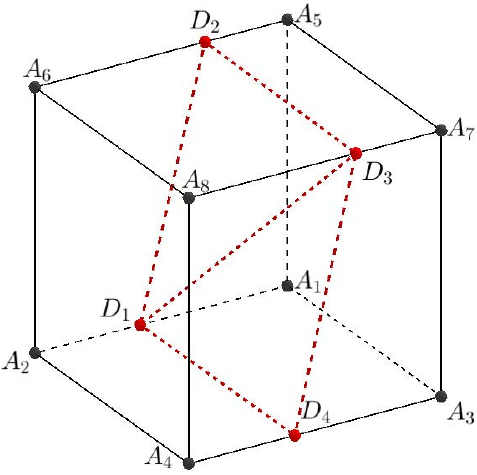}
     \caption{Case 3}
     \label{inter_elem_case3} 
\end{subfigure}
 \begin{subfigure}{.25\textwidth}
     \includegraphics[width=1.5in]{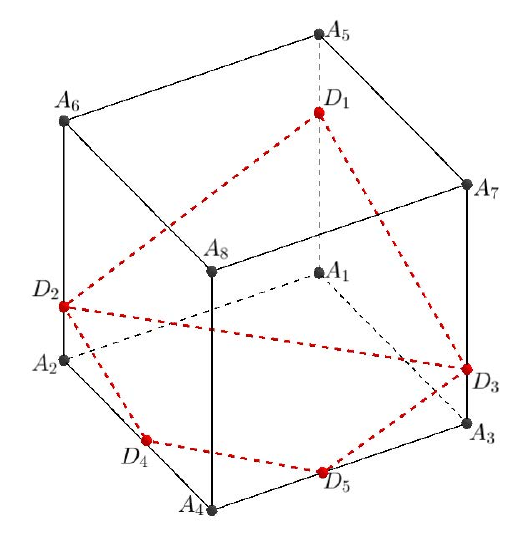}
     \caption{Case 4}
      \label{inter_elem_case4} 
\end{subfigure}
\begin{subfigure}{.25\textwidth}
     \includegraphics[width=1.5in]{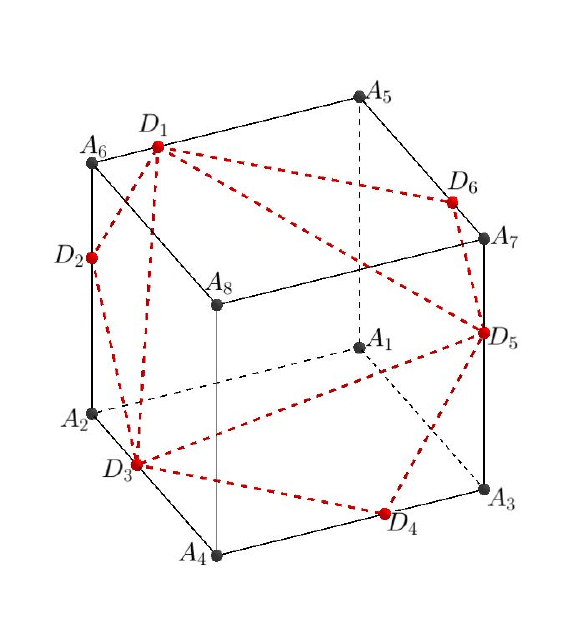}
     \caption{Case 5}
      \label{inter_elem_case5} 
\end{subfigure}
     \caption{Possible Interface Element Configuration }
  \label{fig:subfig_interf} 
\end{figure}
The key idea of the procedure proposed above is that the maximum angle of $K_T = \bigtriangleup D_{j_1}D_{j_2}D_{j_3}$ is bounded above away from $\pi$, and we will show that this maximum angle feature guarantees that the constructed plane
$\tau(T)$ can approximate the surface $\Gamma \cap T$ with a sufficient accuracy.

First, we use the results in \cite{2017ChenWeiWen} to derive bounds of the the maximum angle of $K_T = \bigtriangleup D_{j_1}D_{j_2}D_{j_3}$ in the following lemma.

\begin{lemma}
\label{lem_max_angle}
For each interface element $T\in\mathcal{T}^i_h$, the maximum angle of the triangle $K_T$ described above is bounded by $135^{\circ}$ regardless of the interface location.
\end{lemma}
\begin{proof}
We refer readers to a detailed discussion of the maximum angle associated to a cubic interface element given by Remark 3.4 in \cite{2017ChenWeiWen}. Here, we only recall the possible maximum angles of the triangle $K_T$ for the cases in Figure \ref{fig:subfig_interf}:
\begin{itemize}[leftmargin=45pt]
  \item[Case 1.] $\angle D_3D_1D_2$, $\angle D_1D_2D_3$, $\angle D_2D_3D_1\leqslant 90^{\circ}$.
  \item[Case 2.] $\angle D_4D_1D_2\leqslant 90^{\circ}$, $D_1D_2D_4\leqslant135^{\circ}$, $\angle D_1D_4D_2\leqslant 63.6780^{\circ}$.
  \item[Case 3.] $\angle D_3D_1D_4$, $\angle D_1D_3D_4 \leqslant 63.6780^{\circ}$, $\angle D_1D_4D_3 \leqslant 101.5370^{\circ}$.
  \item[Case 4.] $\angle D_1D_2D_3$, $\angle D_2D_3D_1 \leqslant 63.6780^{\circ}$, $\angle D_3D_1D_2 \leqslant 101.5370^{\circ}$.
  \item[Case 5.] $\angle D_5D_1D_3$, $\angle D_1D_3D_5$, $\angle D_3D_5D_1 \leqslant 101.5370^{\circ}$.
\end{itemize}
\end{proof}
Consider an auxiliary function $f(t)=4\sin{(135^{\circ})}(1-\sqrt{3}t)^4-3t^2-4\sqrt{3}t$ and let $\epsilon_0$ be the smallest positive zero of $f$. Then $\epsilon_0\approx0.134$, and $f(t)$ is positive and decreasing over $(0,\epsilon_0)$. Using this auxiliary function, we can derive a bound for $\alpha_{\Gamma}(K_T)$ in the following lemma.
\commentout{
Now, we recall $\alpha_{\Gamma}(\triangle_T)$ is the maximum angle between $\bar{\mathbf{ n}}(\triangle_T)$ and the normal vectors to $\mathcal{S}_{\Gamma}(\triangle_T)$. Define a function $f(t)=4\sin{(135^{\circ})}(1-\sqrt{3}t)^4-3t^2-4\sqrt{3}t$. Let $\epsilon_0\approx0.134$ be a zero of $f(t)$. One can verify that $f(t)$ is decreasing over $(0,\epsilon_0)$. Then we can establish an estimate of $\alpha_{\Gamma}(\triangle_T)$ in the following.
}
\begin{lemma}
\label{lem_element_angle_est}
Let $\mathcal{T}_h$ be a Cartesian mesh satisfying \textbf{Assumptions} (\textbf{H1})-(\textbf{H3}) sufficiently fine such that
$\kappa h \leqslant \epsilon< \epsilon_0$ for a certain positive number $\epsilon$. If $T\in\mathcal{T}^i_h$ is such that
the projection from $\mathcal{S}_{\Gamma}(K_T)$ onto $K_T$ is bijective, then
\begin{equation}
\label{element_angle_est_eq1}
\sin{(\alpha_{\Gamma}(K_T))} \leqslant C(\epsilon) \kappa h~~ \text{with} ~~ C(\epsilon) = \sqrt{3}\left( \frac{1}{1 - \sqrt{3}\epsilon} + \frac{4 + \sqrt{3}\epsilon }{ f(\epsilon) } \right) > 0.
\end{equation}
\end{lemma}
\begin{proof}
The proof follows basically from Theorem \ref{thm_angle_est} applied to $K= K_T$. We first need to verify the conditions of Theorem \ref{thm_angle_est}. Let $T$ be an arbitrary interface element.
We note that the diameter of $T$ is $\sqrt{3}h$; hence,  by \eqref{l(K)} and (\textbf{H1}), $l(K_T)\leqslant \sqrt{3}h < r_{\Gamma}$ and $\mathcal{H}_{\Gamma}(K_T)\leqslant \sqrt{3}h$. According to \eqref{lambda(K)} and Lemma \ref{lem_max_angle}, we have $\lambda(K_T)\geqslant \sin(135^{\circ})$. Then, the condition $\kappa h \leqslant \epsilon < \epsilon_0$ implies
\begin{equation}
\begin{split}
\label{element_angle_est_eq2}
4\lambda(K_T)(1 - \kappa\mathcal{H}_{\Gamma}(K_T))^4 - \kappa^2l(K_T)^2 - 4\kappa l(K_T) \geqslant 4\sin{(135^{\circ})}( 1 - \sqrt{3}\epsilon )^4 - 3\epsilon^2  - 4\sqrt{3}\epsilon > f(\epsilon_0) = 0.
\end{split}
\end{equation}
Therefore, estimate \eqref{angle_est_eq0} given in Theorem \ref{thm_angle_est} holds for $K = K_T$ so that we have
\begin{equation}
\label{element_angle_est_eq3}
\sin{(\alpha_{\Gamma}(K_T))} \leqslant \kappa h \sqrt{3}\left( \frac{1}{1 - \sqrt{3}\epsilon} + \frac{4 + \sqrt{3}\epsilon }{4\sin{(135^{\circ})}( 1 - \sqrt{3}\epsilon )^4 - 3\epsilon^2  - 4\sqrt{3}\epsilon} \right) = C(\epsilon) \kappa h.
\end{equation}
\end{proof}

\begin{rem}
We note that the function $C(\epsilon)$ is also increasing over $(0,\epsilon_0)$.
Furthermore, because of the orientation $\alpha_{\Gamma}(K_T)\in [0,\pi/2]$, \eqref{element_angle_est_eq1} controls the size of the angle $\alpha_{\Gamma}(K_T)$, i.e., how much the normal vectors of $\mathcal{S}_{\Gamma}(K_T)$ can vary from the normal vector $\bar{\mathbf{ n}}(K_T)$ of the triangle $K_T$, and this actually quantifies the flatness of $\mathcal{S}_{\Gamma}(K_T)$.
\end{rem}

We are now ready to investigate how well the plane $\tau(T)$ determined by the triangle $K_T$ can approximate
$\Gamma \cap T$ or $\Gamma \cap \omega_T$ on an interface element $T\in \mathcal{T}_h^i$.
For this purpose, we further denote $\alpha_T$($\alpha_{\omega_T}$) as the maximum angle between $\bar{\mathbf{ n}}(K_T)$ and normal vectors to $\Gamma\cap T$($\Gamma\cap\omega_T$). By definition, we have $\alpha_{\Gamma}(K_T)\leqslant\alpha_T\leqslant\alpha_{\omega_T}$.

\begin{figure}[H]
\centering
\begin{subfigure}{.25\textwidth}
     \includegraphics[width=2in]{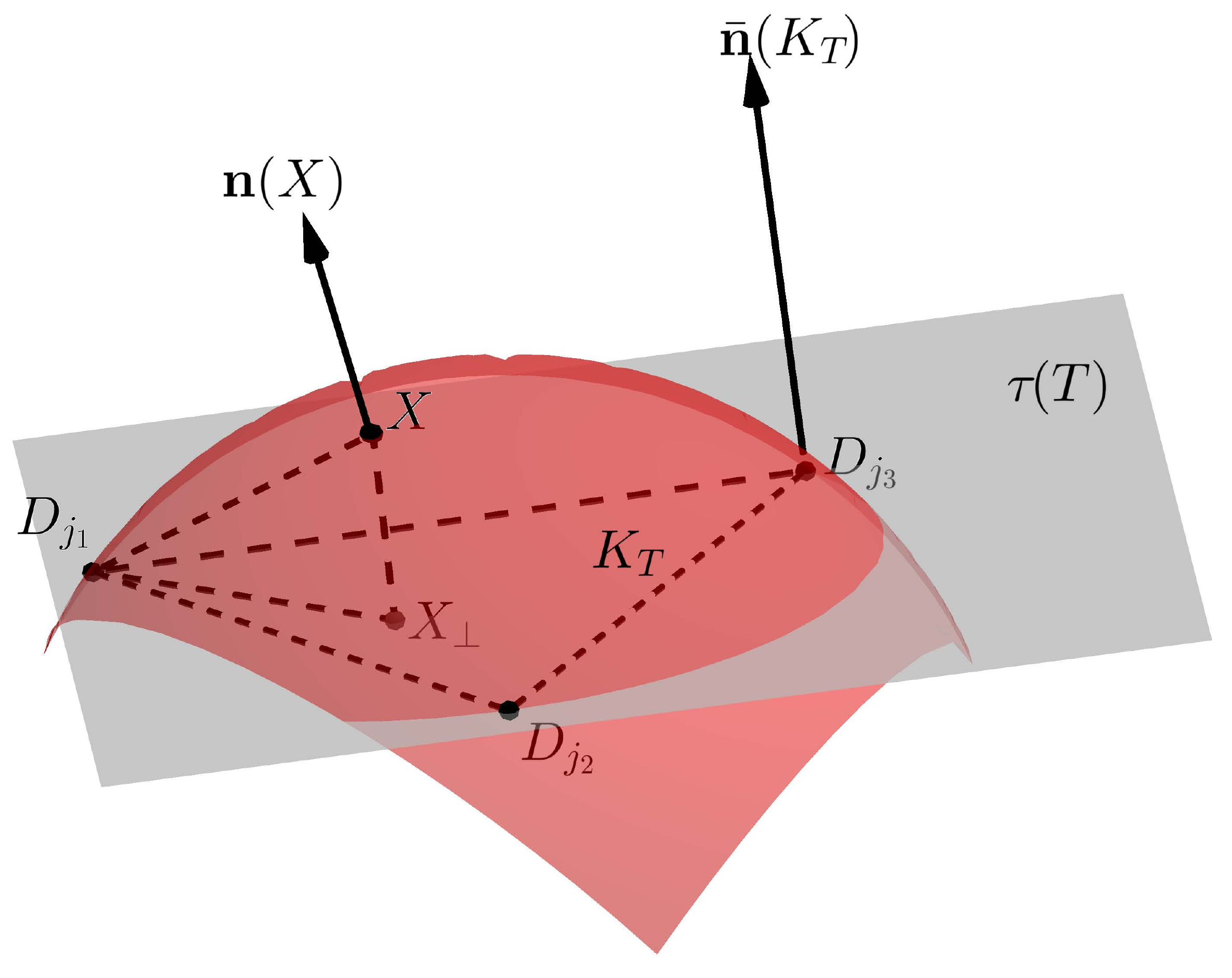}
     \caption{$X$ belongs to $\mathcal{S}_{\Gamma}(K_T)$}
     \label{surface_est_1} 
\end{subfigure}
~~~~~~~~~~~~~~~~
\begin{subfigure}{.25\textwidth}
     \includegraphics[width=2in]{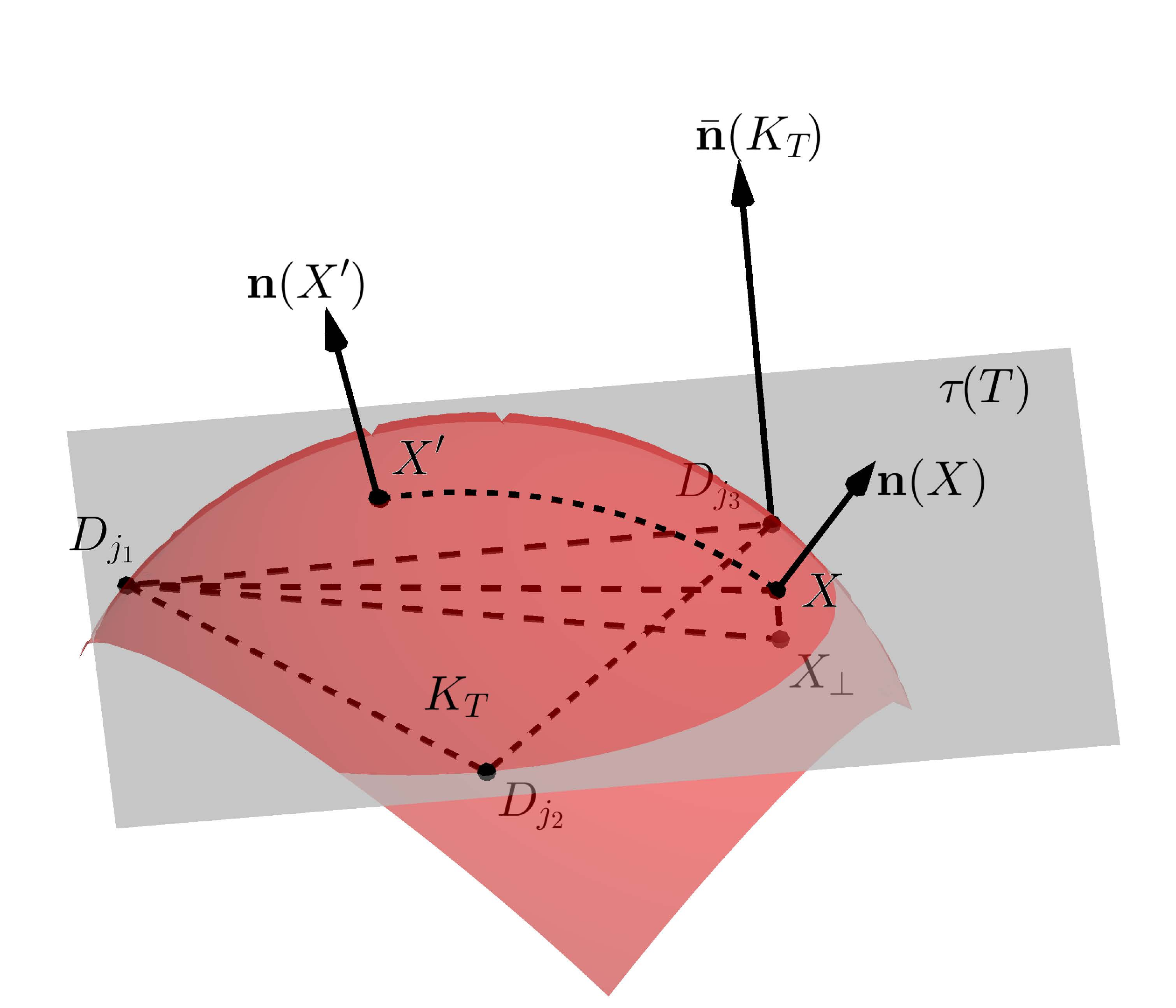}
     \caption{$X$ belongs to $\Gamma\cap T$ but not on $\mathcal{S}_{\Gamma}(K_T)$}
     \label{surface_est_2} 
\end{subfigure}
     \caption{Geometry of an interface surface and the plane $\tau$}
  \label{fig:surface_geo} 
\end{figure}

\begin{thm}
\label{lem_interf_element_est}
Let $\mathcal{T}_h$ be  a Cartesian mesh whose mesh size is small enough such that the \textbf{Assumptions} (\textbf{H1})-(\textbf{H3}) hold. If $T \in \mathcal{T}_h^i$ is such that $\kappa h\leqslant \epsilon<0.0836$ and
$\alpha_T\leqslant\pi/2$ (or $\kappa h\leqslant \epsilon<0.0288$ and
$\alpha_{\omega_T}\leqslant\pi/2$), then there exist constants $C$ depending only on $\epsilon$ such that the following estimates hold for every point $X\in\Gamma\cap T$ (or every point $X\in\Gamma\cap \omega_T$):
\commentout{
for each interface element $T$ or $\kappa h\leqslant \epsilon<0.0288$, $\alpha_{\omega_T}\leqslant\pi/2$ for each patch $\omega_T$, then there exist constants $C$ depending only on $\epsilon$ such that, for every point $X\in\Gamma\cap T$ or $X\in \omega_T\cap\Gamma$, the following estimates hold
}
\begin{subequations}
\label{interf_element_est_eq0}
\begin{align}
    & \| X - X_{\bot} \| \leqslant C \kappa h^2, \label{interf_element_est_eq0_dist} \\
    &  \| \mathbf{ n}(X) - \bar{\mathbf{ n}}(K_T) \| \leqslant C \kappa h, \label{interf_element_est_diff} \\
    &  \mathbf{ n}(X) \cdot \bar{\mathbf{ n}}(K_T) \geqslant 1 - C \kappa^2 h^2,  \label{interf_element_est_prod}
\end{align}
\end{subequations}
where $X_{\bot}$ is the projection of $X$ onto $\tau(T)$ and $\mathbf{ n}(X)$ is the normal vector to $\Gamma$ at $X$.
\end{thm}
\begin{proof}
We only prove these estimates for $X\in\Gamma\cap T$ because similar arguments apply to $X\in\Gamma\cap \omega_T$.
Let $\theta(\mathbf{ v}_1,\mathbf{ v}_2)\in[0,\pi]$ be the angle between any two vectors $\mathbf{ v}_1$ and $\mathbf{ v}_2$.
Consider a point $X\in\Gamma\cap T$ and its projection $X_{\bot}$ onto $\tau(T)$. If $X\in\mathcal{S}_{\Gamma}(K_T)$, then $X_{\bot}\in K_T$, as illustrated in Figure \ref{fig:surface_geo}(\subref{surface_est_1}), and $\angle XD_{j_1}X_{\bot} \leqslant \alpha_{\Gamma}(K_T)\leqslant \alpha_T \leqslant \pi/2$. We note that $h \kappa\leqslant \epsilon <\epsilon_0$ and the assumption $\alpha_T\leqslant \pi/2$ implies the projection of $\Gamma\cap T$ on to $\tau(T)$ is bijective; therefore, we can apply Lemma \ref{lem_element_angle_est} to have
\begin{equation}
\label{interf_element_est_eq1}
\| X - X_{\bot} \| = \sin(\angle XD_{j_1}X_{\bot} ) \| X - D_{j_1} \|\leqslant \sin(\alpha_{\Gamma}(K_T)) \sqrt{3}h \leqslant \sqrt{3}C(\epsilon)\kappa h^2.
\end{equation}
For every $X\in (\Gamma \cap T) - \mathcal{S}_{\Gamma}(K_T)$, we consider another point $X'\in\mathcal{S}_{\Gamma}(K_T)$. Then, by Lemma 2 and Lemma 4 in \cite{2002MorvanThibert}, we have
\begin{equation}
\label{interf_element_est_eq2}
\sin( \theta(\mathbf{ n}(X),\mathbf{ n}(X')) )\leqslant \kappa L_{\Gamma}(X,X')\leqslant \frac{\kappa }{1 -\sqrt{3}\kappa h } \| X - X' \| \leqslant \frac{\sqrt{3}}{1-\sqrt{3}\epsilon} \kappa h,
\end{equation}
where $L_{\Gamma}(X, X')$ is the geodesic distance between $X$ and $X'$ as shown by the dashed line on surface in Figure \ref{fig:surface_geo}(\subref{surface_est_2}). Note that $\theta(\mathbf{ n}(X),\bar{\mathbf{ n}}(K_T))\leqslant\pi/2$, $\theta(\mathbf{ n}(X),\mathbf{ n}(X')) \leqslant\pi/2$ and $\theta(\mathbf{ n}(X'),\bar{\mathbf{ n}}(K_T))\leqslant\pi/2$, we use \eqref{interf_element_est_eq2} and \eqref{element_angle_est_eq1} to obtain
\begin{equation}
\label{interf_element_est_eq4}
\sin(\theta(\mathbf{ n}(X),\bar{\mathbf{ n}}(K_T))) \leqslant \sin(\theta(\mathbf{ n}(X),\mathbf{ n}(X')))  +
\sin(\theta(\mathbf{ n}(X'),\bar{\mathbf{ n}}(K_T))) \leqslant \left( \frac{\sqrt{3}}{1 - \sqrt{3}\epsilon} + C(\epsilon) \right)\kappa h.
\end{equation}
which, together with \eqref{element_angle_est_eq1}, implies that
\begin{align}
&\sin(\alpha_T)  \leqslant \sup_{X \in \Gamma \cap T} \sin(\theta(\mathbf{ n}(X),\bar{\mathbf{ n}}(K_T))) \leqslant \left( \frac{\sqrt{3}}{1 - \sqrt{3}\epsilon} + C(\epsilon) \right)\kappa h. \label{interf_element_est_eq6}
\end{align}
Hence, for $X\in (\Gamma \cap T) - \mathcal{S}_{\Gamma}(K_T)$,  following an argument similar to \eqref{interf_element_est_eq1} we have
\begin{equation}
\label{interf_element_est_eq5}
\| X - X_{\bot} \| \leqslant \sin(\alpha_T) \sqrt{3}h \leqslant \left( \frac{3}{1 - \sqrt{3}\epsilon} + \sqrt{3}C(\epsilon) \right)\kappa h^2,
\end{equation}
and \eqref{interf_element_est_eq0_dist} follows from \eqref{interf_element_est_eq4} and \eqref{interf_element_est_eq5}.
Furthermore, we note that \eqref{interf_element_est_eq4} leads to
$$\sin(\theta(\mathbf{ n}(X),\bar{\mathbf{ n}}(K_T))) \leqslant \left( \frac{\sqrt{3}}{1 - \sqrt{3}\epsilon} + C(\epsilon) \right)\kappa h\leqslant \left( \frac{\sqrt{3}}{1 - \sqrt{3}\epsilon} + C(\epsilon) \right)\epsilon <1
$$
for $\epsilon < 0.0836$. Hence $\mathbf{ n}(X)\cdot\bar{\mathbf{ n}}(K_T) = \cos(\theta(\mathbf{ n}(X),\bar{\mathbf{ n}}(K_T))) = \sqrt{1-\sin^2(\theta(\mathbf{ n}(X),\bar{\mathbf{ n}}(K_T)))}$ with \eqref{interf_element_est_eq4} leads to \eqref{interf_element_est_prod}. Besides, $\| \mathbf{ n}(X) - \bar{\mathbf{ n}}(K_T) \|^2 = 2 - 2  \mathbf{ n}(X)\cdot\bar{\mathbf{ n}}(K_T)$ with \eqref{interf_element_est_prod} yields \eqref{interf_element_est_diff}.
\end{proof}


We note that the estimates similar to \eqref{interf_element_est_eq0} have been derived in \cite{2016GuoLin} for 2-D interface elements in terms of curve curvature and mesh size.  Interface elements in 3D have more complicated geometries and the coplanarity issue
is unavoidable. The maximum angle property of the triangle that determines the plane $\tau(T)$ to approximate the interface surface
in each interface element is critical.

\commentout{
\begin{rem}
For linear IFE functions on tetrahedral elements, the coplanar issue also exists when there are four intersection points on one tetrahedral element. This problem was discussed in \cite{2005KafafyLinLinWang} where the authors neglected the intersection point that has the minimal normal distance from the plane made by the remaining three intersection points, and let the approximating plane be the plane spanned by these remaining points. They claimed the chosen plane has the $\mathcal{O}(h^2)$ approximation accuracy to the interface surface. We believe the techniques developed here can be also applied to that case for a rigorous proof.
\end{rem}
}


\section{Trilinear IFE Spaces}
\label{sec:tri_IFE_func}

In this section, we develop trilinear IFE spaces for solving interface problems described by \eqref{model}.
Without loss of generality, we assume $\beta^+\geqslant\beta^-$ in the following discussion.
Denote the space of trilinear polynomials by $\mathbb{Q}_1$. Following the basic idea of the IFE method, on each non-interface element, we simply let the local IFE space be
\begin{equation}
\label{FE_space_loc}
S_h(T) = \mathbb{Q}_1,~~\forall T \in \mathcal{T}^n_h.
\end{equation}

Our major effort is to develop the local IFE space on each interface element $T \in \mathcal{T}^i_h$ in which we have defined
the plane $\tau(T)$ to approximate the interface surface $\Gamma$. Let $L(X)=(X-D_1)\cdot\bar{\mathbf{ n}}$ be such that $L(X)=0$ is an equation of the plane $\tau(T)$ and $\bar{\mathbf{n}}=\bar{\mathbf{ n}}(K_T)$ is the normal of $\tau(T)$.
\commentout{
On each interface element $T$, the key step in our construction procedure is to formulate approximate jump conditions on the plane $\tau$ descried in Section \ref{sec:geometry} associated with the interface surface. For simplicity of presentation, from now on, we employ the notation $\bar{\mathbf{n}}=\bar{\mathbf{ n}}(\triangle_T)$. Define a linear function $L(X)=(X-D_1)\cdot\bar{\mathbf{ n}}$ such that $L(X)=0$ is an equation of the plane $\tau$. Let $F$ be the centroid of the triangle $\triangle_T$.
}
Then, we consider trilinear IFE functions in a piecewise trilinear polynomial format:
\begin{equation}
\label{ife_shape_fun_1}
\phi_T(X) =
\left\{
\begin{aligned}
\phi^{-}_T(X)\in \mathbb{Q}_1 \;\;\;\;\; & \text{if} \;\; X\in T^-, \\
\phi^{+}_T(X)\in \mathbb{Q}_1 \;\;\;\;\; & \text{if} \;\; X\in T^+,
\end{aligned}
\right.
\end{equation}
such that the two polynomial components $\phi^{-}_T$ and $\phi^{+}_T$ satisfy the following approximate jump conditions:
\begin{subequations}
\label{ife_shape_fun_2}
\begin{align}
\phi^{-}_T|_{\tau}=\phi^{+}_T|_{\tau}, ~~ d(\phi^{-}_T)=d(\phi^{+}_T), \label{ife_shape_fun_2_1} \\
\beta^-\nabla\phi^{-}_T(F)\cdot \bar{\mathbf{ n}} = \beta^+\nabla\phi^{+}_T(F)\cdot \bar{\mathbf{ n}},  \label{ife_shape_fun_2_2}
\end{align}
\end{subequations}
in which $F$ is the centroid of the triangle $K_T$ and $d(p)$ denotes the coefficient of the term $xy+yz+xz$ in a trilinear polynomial $p \in \mathbb{Q}_1$. These approximate jump conditions are similar to their counterparts in the 2-D case, i.e., the bilinear IFE functions \cite{2016GuoLin,2008HeLinLin}.

By the approximate jump conditions \eqref{ife_shape_fun_2}, we can consider an extension operator as follows:
\begin{equation}
\label{ext_opera}
\mathcal{C}_T:\mathbb{Q}_1\rightarrow\mathbb{Q}_1, ~~ \text{such that} ~  \phi^-_T = p \in \mathbb{Q}_1  \text{~and~} \phi^+_T= \mathcal{C}_T(p) \in \mathbb{Q}_1 ~ \text{together satisfy \eqref{ife_shape_fun_2}}.
\end{equation}




\begin{thm}
\label{thm_poly_rela}
The operator $\mathcal{C}_T$ is well defined and bijective on $\mathbb{Q}_1$.
\end{thm}
\begin{proof}
Since $\mathcal{C}_T$ is a linear operator and clearly, $\mathcal{C}_T(0)=0$, we know $\mathcal{C}_T$ is well defined. Besides, since $\mathcal{C}_T$ is mapping from the finite dimensional space $\mathbb{Q}_1$ to itself, we only need to show $\mathcal{C}_T$ is injective. By \eqref{ife_shape_fun_2}, $p\in \mathbb{Q}_1$ and its extension $\mathcal{C}_T(p)$ satisfy
\begin{equation}
\label{thm_poly_rela_eq1}
 p(X)  - \mathcal{C}_T(p)(X) = c_0L(X), ~~~ \text{with}~
 c_0 = \left( 1 - \frac{\beta^-}{\beta^+}  \right) \big(\nabla \mathcal{C}_T(p)(F)\big)\cdot\bar{\mathbf{ n}}.
\end{equation}
Hence, according to \eqref{thm_poly_rela_eq1}, $\mathcal{C}_T(p)=0$ must imply $p=0$, which finishes the proof.
\end{proof}
Furthermore, according to Theorem \ref{thm_poly_rela}, we have
\begin{equation}
\label{ext_c_express}
\mathcal{C}_T(p) = p + \left(\frac{\beta^-}{\beta^+} - 1 \right) (\nabla p(F)\cdot \bar{\mathbf{ n}}) L(X) ~ \text{and} ~ \mathcal{C}^{-1}_T(p) = p + \left(\frac{\beta^+}{\beta^-} - 1 \right) (\nabla p(F)\cdot \bar{\mathbf{ n}})L(X), ~
\forall p \in \mathbb{Q}_1.
\end{equation}
Using this extension operator $\mathcal{C}_T$, we define the local IFE space on every interface element as follows:
\begin{equation}
\label{loc_IFE_space}
S_h(T) = \{ \phi_T \text{~is piecewise trilinear such that} ~ \phi_T|_{T^-}=p \in \mathbb{Q}_1 ~ \text{and} ~  \phi_T|_{T^+}= \mathcal{C}_T(p)  \},~~\forall T \in \mathcal{T}_h^i.
\end{equation}
The bijection of $\mathcal{C}_T$ directly implies the dimension of $S_h(T)$ is 8 which is the same as $\mathbb{Q}_1$, i.e., the standard local trilinear finite element space. Unlike some 2-D linear and bilinear IFE spaces in the literature \cite{2008HeLinLin,2004LiLinLinRogers,2015LinLinZhang,2013ZhangTHESIS} whose functions are piecewise polynomials defined according to a linear approximation of the interface $\Gamma$ on each interface element, each trilinear IFE function of $S_h(T)$ by \eqref{loc_IFE_space} is a piecewise polynomial defined according to the interface surface $\Gamma$ itself, not its linear approximation
$\tau(T)$. As pointed out in \cite{2016GuoLin}, IFE functions defined according to the interface itself have some advantages in both computation and analysis, even more so in 3-D. For example, if IFE functions are defined according to the plane $\tau(T)$ on each interface element $T$, then the involved computations have to vary from one interface element to another because
the coplanarity issue requires $\tau(T)$ to be constructed in different ways on different interface elements. Also,
the traces of such an IFE function on an interface face induced from two adjacent interface elements will have different discontinuities which require extra treatments to compute quantities on interface faces such as penalties needed in an IFE scheme based on a DG formulation.

\begin{rem}\label{rem_ext_1}
By definition, each IFE function in $S_h(T)$ is formed by extending
a trilinear polynomial used on $T^-$ to another trilinear polynomial on $T^+$. The choice of the extension from $T^-$ to
$T^+$ follows from the assumption that $\beta^+\geqslant\beta^-$ which is useful for related analysis. Because of
Theorem \ref{thm_poly_rela} and \eqref{ext_c_express}, extension from $T^+$ to $T^-$ can be introduced similarly when
$\beta^-\geqslant\beta^+$ and all results related to the corresponding IFE spaces hold. We also note that the construction of IFE spaces based on some extension operator has already been reported in the literature for the 2D case which can be traced back to \cite{2014AdjeridBenromdhaneLin} and explicitly studied in \cite{2019GuoLin,2016GuzmanSanchezSarkis,2015GuzmanSanchezSarkisP1}.
\end{rem}
\commentout{
\begin{rem}
Unlike some 2-D linear and bilinear IFE spaces in the literature \cite{2008HeLinLin,2004LiLinLinRogers,2015LinLinZhang,2013ZhangTHESIS} whose functions are piecewise polynomials defined according to a linear approximation of the interface $\Gamma$ on each interface element, each trilinear IFE function of $S_h(T)$ by \eqref{loc_IFE_space} is a piecewise polynomial defined according to the interface surface $\Gamma$ itself, not its linear approximation
$\tau(T)$. As pointed out in \cite{2016GuoLin}, IFE functions defined according to the interface itself have some advantages in both computation and analysis, even more so in 3-D. For example, if IFE functions are defined according the plane $\tau(T)$ on each interface element $T$, then the involved computations procedures have to vary from one interface element to another because
the coplanarity issue requires $\tau(T)$ to be constructed in different ways on different interface elements. Also,
the traces of an IFE function on an interface face induced from two adjacent interface elements will have different discontinuities which require extra treatments to compute quantities on interface faces and edges such as penalties needed in the PPIFE scheme to be
presented later.
\end{rem}
}
\commentout{
The relationships in \eqref{ext_c_express} suggest that we can also consider an extension operator $\mathcal{C}_T~:~\mathbb{Q}_1\rightarrow\mathbb{Q}_1$ as $\mathcal{C}_T(\phi^-_T)=\phi^+_T$, i.e., from $-$ to $+$, and all the results above can be proved similarly. Here, we choose the definition \eqref{thm_poly_rela_eq1} because $\beta^+\geqslant\beta^-$ which can facilitate a simple presentation of analysis in the following.
}

We now discuss how to chose IFE shape functions from the local IFE space \eqref{loc_IFE_space} according to some desirable features. In particular, we consider the IFE shape functions with the Lagrange type degrees freedom, i.e., an IFE function identified by its nodal values. For this purpose, on each interface element $T$ with vertices $A_i, 1 \leq i \leq 8$, we introduce the index set $\mathcal{I}=\{1,2,\cdots,8\}$ and its sub-sets $\mathcal{I}^s=\{i\in\mathcal{I}:A_i\in T^s \}$, $s=\pm$. Then, we enforce the nodal value condition for IFE functions:
\begin{equation}
\label{ife_shape_fun_3}
\phi_{T}(A_i)=v_i, \;\;\;\; v_i\in\mathbb{R}, ~~ i \in \mathcal{I}.
\end{equation}
Let $\psi_{i,T}$, $i\in\mathcal{I}$ be the standard Lagrange trilinear shape function
associated with the vertex $A_i, i\in\mathcal{I}$. Then, following
the idea in \cite{2016GuoLin} and the assumption that $\beta^+ \geq \beta^-$, we use the second equation in \eqref{ext_c_express} to express $\phi_T$ on $T^-$ as follows:
\begin{equation}
\label{ife_shape_fun_4}
\phi_T(X) =
\begin{cases}
 \phi^{-}_T(X) = \mathcal{C}_T^{-1}(\phi^+_T) = \phi^{+}_T(X)+c_0L(X), & \text{if} \;\; X\in T^-, \\
 \phi^{+}_T(X)  = \sum_{i\in\mathcal{I}^-}c_i\psi_{i,T}(X)+\sum_{i\in\mathcal{I}^+}v_i\psi_{i,T}(X), & \text{if} \;\; X\in T^+,
\end{cases}
\end{equation}
in which, according to \eqref{thm_poly_rela_eq1}, the constant $c_0$ has the following expression:
\begin{equation}
\label{ife_shape_fun_5}
c_0=
\mu \left( \sum_{i\in\mathcal{I}^-}c_i\nabla\psi_{i,T}(F)\cdot \bar{\mathbf{ n}} + \sum_{i\in\mathcal{I}^+}v_i\nabla\psi_{i,T}(F)\cdot \bar{\mathbf{ n}} \right),
\end{equation}
where $\mu=\frac{\beta^+}{\beta^-}-1\geqslant0$. Next, by enforcing the nodal value condition \eqref{ife_shape_fun_3} at the nodes $A_i\in T^-$, we obtain a linear system for the the unknown coefficients $\mathbf{ c}=(c_i)_{i\in\mathcal{I}^-}$:
\begin{equation}
\label{ife_shape_fun_6}
(I + \mu \,\bfdelta \bfgamma^T )\bfc = \bfb,
\end{equation}
where
\begin{equation}
\label{ife_shape_fun_7}
\bfgamma = \left( \nabla\psi_{i,T}(F)\cdot\bar{ \mathbf{ n}} \right)_{i\in\mathcal{I}^-}, ~~ \bfdelta=\left( L(A_i) \right)_{i\in\mathcal{I}^-} ~~ \text{and} ~~
\bfb =
\left(
v_i- \mu L(A_i) \sum_{j\in\mathcal{I}^+} \nabla\psi_{j,T}(F)\cdot \bar{\mathbf{ n}} v_j
 \right)_{i\in\mathcal{I}^-}.
\end{equation}
are all column vectors in $\mathbb{R}^{|\mathcal{I}^-|}$ whose values are known. We note that the matrix in \eqref{ife_shape_fun_6} is in a Sherman-Morrison type, and its structure described by \eqref{ife_shape_fun_6}-\eqref{ife_shape_fun_7} is exactly the same as the one for the 2-D IFE shape functions reported in \cite{2016GuoLin}. 

Next, we analyze the solvability of the linear system in \eqref{ife_shape_fun_6}, referred as the unisolvence of Lagrange IFE shape functions. Here, we emphasize that the unisolvence should be independent of the interface element configuration and the coefficients $\beta^{\pm}$. We need the following lemma.

\begin{lemma}
\label{lemma_gammadelta_01}
For all the cubic interface elements as shown in Figure \ref{fig:subfig_cub}, we have
$
\bfgamma^T\bfdelta \in [0,1]
$.
Furthermore, there holds
\begin{equation}
\label{lemma_gammadelta_01_eq1}
h^{-1}\| \bfdelta \|_{\infty} \leqslant 7.43 \bfgamma^T\bfdelta,
\end{equation}
where $\|\cdot\|_{\infty}$ denotes the infinity norm for a vector.
\end{lemma}
\begin{proof}
The proof follows from direct calculations, see Appendix \ref{App_A2} for more details.
\end{proof}

\begin{thm}[Unisolvence]
\label{unisolv_IFE_shape}
On each interface element $T$, regardless of the interface location and coefficients $\beta^{\pm}$, given any vector $\mathbf{ v}=(v_i)_{i\in\mathcal{I}}^T\in\mathbb{R}^8$, there exists a unique IFE function $\phi_T\in S_h(T)$ that satisfies the nodal value condition \eqref{ife_shape_fun_3}.
\end{thm}
\begin{proof}
Using Lemma \ref{lemma_gammadelta_01} and recalling $\mu=\frac{\beta^+}{\beta^-}-1$ with $\beta^+\geqslant\beta^-$, we obtain
\begin{equation}
\label{unisolv_IFE_shape_eq1}
1 + \mu \bfgamma^T\bfdelta \geqslant \min\bigg\{ 1, \frac{\beta^+}{\beta^-} \bigg\}\geqslant 1 >0.
\end{equation}
Therefore, by the well known Sherman-Morrison formula, the matrix in \eqref{ife_shape_fun_6} is nonsingular. Hence,
all the coefficients $\mathbf{ c}=(c_i)_{i\in\mathcal{I}^-}$ and $c_0$ of $\phi_T\in S_h(T)$ are uniquely determined
by $\mathbf{ v}=(v_i)_{i\in\mathcal{I}}^T\in\mathbb{R}^8$.
\end{proof}
In addition, by the Sherman-Morrison formula, we have the following explicit formulas for the unknown coefficients $\mathbf{ c}$ and $c_0$:
\begin{equation}
\label{c_c0_formula}
   \mathbf{ c} = \mathbf{ v}^- - \frac{\mu \Xi \bfdelta }{ 1 + \mu \bfgamma^T \bfdelta } ~~~ \text{and} ~~~  c_0 = \frac{\mu \Xi }{ 1 + \mu \bfgamma^T\bfdelta },
\end{equation}
where $\mathbf{ v}^- = (v_i)^T_{i\in\mathcal{I}^-}\in\mathbb{R}^{|\mathcal{I}^-|}$ and $\Xi = \sum_{i\in\mathcal{I}}v_i\nabla\psi_{i,T}(F)\cdot\bar{\mathbf{ n}}$. By Theorem \ref{unisolv_IFE_shape}, we can let $\mathbf{ v}$ be the unit vector $\mathbf{ e}_i$, $i\in\mathcal{I}$, in $\mathbb{R}^8$ and define the corresponding Lagrange IFE shape function, $\phi_{i,T}$ such that
\begin{equation}
\label{IFE_shape_fun_delta}
\phi_{i,T}(A_j) = \delta_{i,j}, ~~~ i,j\in\mathcal{I}.
\end{equation}
The Lagrange IFE shape functions provide an alternative description of the local IFE space \eqref{loc_IFE_space}:
\begin{equation}
\label{loc_IFE_space_span}
S_h(T) = \text{Spna} \{ \phi_{i,T}~:~ i\in\mathcal{I} \}.
\end{equation}
We note that the explicit formulas \eqref{c_c0_formula} facilitate the implementation of the Lagrange IFE shape functions because
they depend straightforwardly on the centroid $F$ and normal of the triangle $K_T$ formed inside each interface element $T$ according to the procedure developed in Section \ref{sec:geometry}.

\commentout{
\begin{rem}
\label{ife_shape_feature}
We emphasize that in the proposed construction of the IFE spaces \eqref{loc_IFE_space} and the related shape functions \eqref{IFE_shape_fun_delta}, the approximate plane $\tau$ or its intersection $\tau\cap T$ does not have to be explicitly generated. Besides, the structure of the explicit formulas \eqref{c_c0_formula} do not rely on the interface element configurations shown in Figure \ref{fig:subfig_interf}. Actually, when applying \eqref{c_c0_formula} for computing IFE shape functions, one only needs to input the data (related to the interface) of the point $F$, the normal vector $\bar{\mathbf{ n}}$, and location of element vertices relative to interface surface. These features make the proposed method easy for implementation.
\end{rem}
}

\commentout{
\begin{rem}
\label{surface_partition}
In the proposed IFE shape functions, the plane $\tau$ is only used to formulate the approximate jump conditions \eqref{ife_shape_fun_2} and our IFE functions \eqref{ife_shape_fun_1} are still piecewise polynomials partitioned by the original interface surface $\Gamma$ instead of the plane $\tau$. We note that most works in the literature for low-degree IFE functions in the 2-D case only consider the IFE functions partitioned by the related linear approximation such as \cite{2008HeLinLin,2004LiLinLinRogers,2015LinLinZhang,2013ZhangTHESIS} because of the resulted local $H^1$ regularity on interface elements. However, this partitioning approach causes some extra difficulties in the 3-D case. Firstly, the location of element vertices relative to the interface surface and the plane $\tau$ may be different, and this yields obscurity in both the analysis and computation. Secondly, the trace of IFE functions partitioned by the plane $\tau$ from the two neighborhood elements on each interface face may have different discontinuities which requires extra work to compute penalties needed in the proposed PPIFE scheme (discussed in the following) on interface faces.
\end{rem}
}

\begin{rem}
\label{rem_IFEshape_unisol}
By a comparison of \emph{(2.4)} in \cite{2010VallaghePapadopoulo} and \eqref{thm_poly_rela_eq1} above, we note that the proposed construction procedure for the Lagrange type IFE shape functions is similar to the one reported in \cite{2010VallaghePapadopoulo}, but essential differences exist. The method in \cite{2010VallaghePapadopoulo} relies on a level set representation of the interface surface and uses the level set function itself in the construction instead of its linear approximation. A theorem in \cite{2010VallaghePapadopoulo} guarantees the unisolvence of the Lagrange IFE shape functions on the admissible coefficient set $\{(\beta^-,\beta^+)~:~ \beta^->0,~\beta^+>0\}$ sans a subset of measure zero. In contrast, the proposed construction procedure does not have this limitation because, by Theorem \ref{unisolv_IFE_shape} above, the unisolvence of the IFE shape functions always holds regardless of interface location and admissible coefficients $\beta^{\pm}$. Furthermore, properties such as \eqref{unisolv_IFE_shape_eq1} and \eqref{c_c0_formula} in the proposed construction procedure are useful in both the analysis and implementation.
\end{rem}

The local IFE space on interface elements can be employed to form global IFE spaces over the whole solution domain
$\Omega$ with desirable features. For example, the following global IFE space can be used in an IFE method based on the DG formulation:
\begin{align}
S_h(\Omega)=&\left\{ v\in L^2(\Omega):v|_{T}\in S_h(T) \text{~defined by \eqref{FE_space_loc} or \eqref{loc_IFE_space}~~}
\forall T \in \mathcal{T}_h \right\}. \label{glob_IFE_space_DG}
\end{align}
Another example of global IFE spaces is
\begin{equation}
\begin{split}
\label{glob_IFE_space}
S_h(\Omega)=&\left\{ v\in L^2(\Omega):v|_{T}\in S_h(T) \text{~defined by \eqref{FE_space_loc} or \eqref{loc_IFE_space}~~}
\forall T \in \mathcal{T}_h,  \right. \\
& \hspace{0.2in} \left. v|_{T_1}(X)=v|_{T_2}(X)~\forall X \in \mathcal{N}_h, \forall\,T_1, T_2 \in \mathcal{T}_h \text{~such that~} X \in T_1\cap T_2 \right\}.
\end{split}
\end{equation}
The global IFE functions in $S_h(\Omega)$ defined by \eqref{glob_IFE_space} are continuous at all the mesh nodes and all the non-interface edges, and this global IFE space $S_h(\Omega)$ is isomorphic to the standard trilinear finite element space defined on the same mesh, in terms of the number and the location of their global degrees of freedom. We will demonstrate how this global IFE space can be used in a PPIFE scheme for solving the interface problems \eqref{model}.

\commentout{
To end this section, we present the partially penalized IFE (PPIFE) scheme to solve the interface problem \eqref{model} based on the Lagrange IFE functions derived above. We consider a underline space
\begin{equation}
\begin{split}
\label{underline_spa}
V_h(\Omega) = &\left\{ v\in L^2(\Omega)~:~ v\in H^2(T)~\forall T\in\mathcal{T}^n_h, ~ v\in H^2(T^{\pm})~\forall T\in\mathcal{T}^i_h; \right. \\
& \hspace{0.2in} \left. \text{and} ~ v ~ \text{is continuous at each} ~ X\in\mathcal{N}_h ~ \text{and each} ~ e\in\mathcal{E}^n_h \right\}.
\end{split}
\end{equation}
The IFE shape functions \eqref{FE_space_loc} and \eqref{loc_IFE_space_span} enable us to define the global IFE space as a subspace of $V_h(\Omega)$:
\begin{equation}
\begin{split}
\label{glob_IFE_space}
S_h(\Omega)=&\left\{ v\in L^2(\Omega):v|_{T}\in S_h(T); \right. \\
& \hspace{0.2in} \left. v|_{T_1}(X)=v|_{T_2}(X)~\forall X \in \mathcal{N}_h, \forall\,T_1, T_2 \in \mathcal{T}_h \text{~such that~} X \in T_1\cap T_2 \right\}.
\end{split}
\end{equation}
We highlight that the global IFE functions in $S_h(\Omega)$ are continuous at all the mesh nodes and all the non-interface edges, but may not be continuous across the interface edges. This feature makes the global IFE space $S_h(\Omega)$ isomorphic to the standard trilinear finite element space defined on the same mesh. Thus, the location and total number of the degrees of freedom are fixed independent of the interface location. Now, we define a bilinear form $a:V_h(\Omega)\times V_h(\Omega)\rightarrow\mathbb{R}$ such that
\begin{equation}
\begin{split}
\label{ppife_1}
a_h(u,v) =& \sum_{T\in\mathcal{T}_h} \int_T \beta \nabla u \cdot \nabla v dX \\
 - & \sum_{F\in\mathring{\mathcal{F}}_h} \int_F \{ \beta \nabla u\cdot \mathbf{ n} \} [v] ds  + \epsilon \sum_{F\in\mathring{\mathcal{F}}_h} \int_F \{ \beta \nabla v\cdot \mathbf{ n} \} [u] ds + \sum_{F\in\mathring{\mathcal{F}}_h} \frac{\sigma^0}{|F|} \int_F [u]\,[v] ds,\\
\end{split}
\end{equation}
and a linear form $L: V_h(\Omega)\rightarrow \mathbb{R}$ such that 
\begin{equation}
\label{ppife_2}
L(v) = \int_{\Omega} f v dX. 
\end{equation}
Then, the proposed IFE method is to find $u_h\in S_h(\Omega)$ such that
\begin{equation}
\label{ppife}
a_h(u_h,v_h) = L(v_h), ~~~ \forall v_h\in S_h(\Omega).
\end{equation}
The similar scheme was already used in \cite{2015LinLinZhang} for 2-D interface problems. The terms defined on interface faces in \eqref{ppife_1} are essentially interior penalties \cite{1982Arnold,1976DouglasDupont} which are used to handle the discontinuities across the interface faces. Because of this similarity, we call \eqref{ppife_1}-\eqref{ppife} symmetric, non-symmetric and incomplete PPIFE (SPPIFE, NPPIFE and IPPIFE) methods for $\epsilon=-1,1$ and $0$, respectively.
}


\section{Approximation Capabilities of Trilinear IFE Spaces}
\label{sec:approximation}

In this section, we will prove that the trilinear IFE spaces developed in the last section have the optimal approximation capability.
The analysis approach is in the spirit of \cite{2019GuoLin,2015GuzmanSanchezSarkisP1}. In the discussion from now, we let $u^+_E\in H^2(\Omega)$ and $u^-_E\in H^2(\Omega)$ be the Sobolev extensions of the components $u^+\in H^2(\Omega^+)$ and $u^-\in H^2(\Omega^-)$ of every function $u\in PH^2(\Omega)$, respectively, such that
\begin{equation}
\label{sobolev_ext}
\| u^s_E \|_{H^2(\Omega)} \leqslant C_E \| u^s \|_{H^2(\Omega^s)}, ~~~ s=\pm,
\end{equation}
for some constant $C_E$. For each interface element $T$, we let $\tau_{\omega_T}=\omega_T \cap \tau(T)$,  $\Gamma_{\omega_T}=\omega_T \cap \Gamma$ where $\omega_T$ is the patch around $T$,
and let $\lambda_{\omega_T}$ be the radius of the largest circle on the plane $\tau(T)$ inscribed inside $\tau_{\omega_T}$, see
the illustration in Figure \ref{fig:patch} where  the cube with solid lines is $T \in \mathcal{T}_h^i$ and the cube with dashed lines is
its patch $\omega_T$. Since $T$ is the center element of its patch, there exists a constant $\delta>0$ independent of interface location such that
\begin{equation}
\label{tau_lambda_diam}
|\lambda_{\omega_T}| \geqslant \delta  h.
\end{equation}
For simplicity's sake, all the generic constants $C$ used in the following discussion are independent of interface location, mesh size $h$ and coefficients $\beta^{\pm}$. We start from two lemmas for properties on $\tau_{\omega_T}$. The first one is about
a kind of trace inequality and similar results have been used in the unfitted mesh finite element methods \cite{2015BurmanClaus,2019GuoLin,2002HansboHansbo}.
\begin{lemma}
\label{lem_trace_inequa_pyd}
There exists a constant $C$ such that
\begin{equation}
\label{trace_inequa_pyd_eq0}
\| v \|_{L^2(\tau_{\omega_T})} \leqslant C \left( h^{-1/2} \| v \|_{L^2(\omega_T)} + h^{1/2} \| \nabla v \|_{L^2(\tau_{\omega_T})} \right) ~~~ \forall v\in H^1(\omega_T),~~\forall T \in \mathcal{T}_h^i.
\end{equation}
\end{lemma}
\begin{proof}
It is basically the application of Lemma 3.2 in \cite{2016WangXiaoXu} onto $\omega_T$ and its subset $\tau_{\omega_T}$.
\end{proof}
The second lemma is for a kind of inverse inequality.
\begin{lemma}
\label{pt_tau}
There exists a constant $C$ such that
\begin{equation}
\label{pt_tau_eq0}
\|  p(X_0) \| \leqslant C  h^{-1} \| p \|_{L^2(\tau_{\omega_T})}, ~~~ \forall p\in \mathbb{Q}_1, ~~\forall X_0\in\tau_{\omega_T},
~~\forall T \in \mathcal{T}_h^i.
\end{equation}
\end{lemma}
\begin{proof}
Note that $\tau_{\omega_T}$ is a convex polygon with $N_e\leqslant 6$ edges. Connecting $X_0$ and the vertices of $\tau_{\omega_T}$, we obtain $N_e$ triangles denoted by $\triangle_i$, $i=1,\cdots,N_e$. Without loss of generality, we assume $|\triangle_1|\geqslant|\triangle_2|\geqslant\cdots\geqslant|\triangle_{N_e}|$. From \eqref{tau_lambda_diam}, we have $N_e|\triangle_1|\geqslant\sum_{i=1}^{N_e}|\triangle_i|=|\tau_{\omega_T}|\geqslant \pi \lambda^2_{\omega_T} \geqslant \pi\delta^2h^2$. Then, on $|\triangle_1|$, we apply the standard trace inequality for polynomials \cite{2003WarburtonHesthaven} to obtain
\begin{equation}
\label{pt_tau_eq1}
\|p(X_0)\|\leqslant C |\triangle_1|^{-1/2} \| p \|_{L^2(\triangle_1)} \leqslant C \sqrt{ |N_e|/|\tau_{\omega_T}| } \|p\|_{L^2(\tau_{\omega_T})} \leqslant C h^{-1} \|p\|_{L^2(\tau_{\omega_T})},
\end{equation}
which finishes the proof.
\end{proof}
Then, we derive two estimates for the Sobolev extensions $u^{\pm}_E$ on $\tau_{\omega_T}$.
\begin{lemma}
\label{lem_upm_tau}
Let $\mathcal{T}_h$ be a Cartesian mesh whose mesh size is small enough such that the \textbf{Assumptions} (\textbf{H1})-(\textbf{H3}) hold and conditions in Theorem \ref{lem_interf_element_est} are satisfied for each patch $\omega_T, \forall T \in \mathcal{T}_h^i$.
Then the following estimates hold for every $u\in PH^2(\Omega)$ on each $T\in \mathcal{T}^i_h$:
\vspace{-0.2in}
\begin{subequations}
\label{upm_tau_eq0}
\begin{align}
   & \| u^+_E - u^-_E \|_{L^2(\tau_{\omega_T})} \leqslant C \left( \sum_{s=\pm}  h^{3/2} | u^s_E |_{H^1(\omega_T)}  +   h^{5/2} | u^s_E |_{H^2(\omega_T)}  \right),  \label{upm_tau_eq01}    \\
   &  \| \beta^+ \nabla u^+_E\cdot\bar{\mathbf{ n}} - \beta^- \nabla u^-_E\cdot\bar{\mathbf{ n}} \|_{L^2(\tau_{\omega_T})} \leqslant C \left( \sum_{s=\pm}  h^{1/2} \beta^s | u^s_E |_{H^1(\omega_T)} +  h \beta^s | u^s_E |_{H^2(\omega_T)} \right).  \label{upm_tau_eq02}
\end{align}
\end{subequations}
\end{lemma}
\begin{proof}
See the Appendix \ref{App_A1}.
\end{proof}

\begin{figure}[h]
\centering
\begin{minipage}{.4\textwidth}
  \centering
  \includegraphics[width=0.6\textwidth]{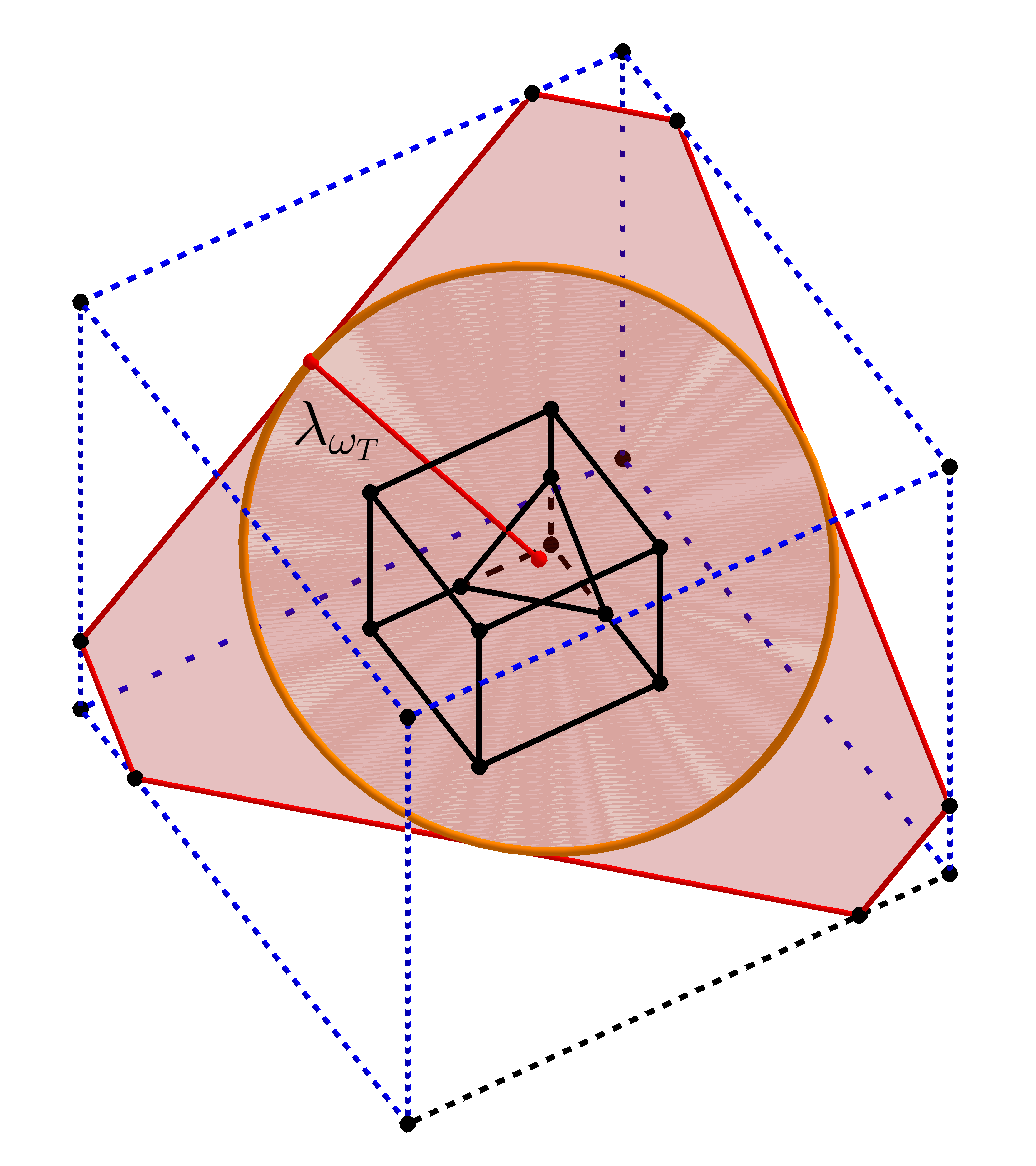}
  \caption{Patch of an interface element}
  \label{fig:patch}
\end{minipage}
\hspace{1cm}
\begin{minipage}{.3\textwidth}
  \centering
  \includegraphics[width=0.9\textwidth]{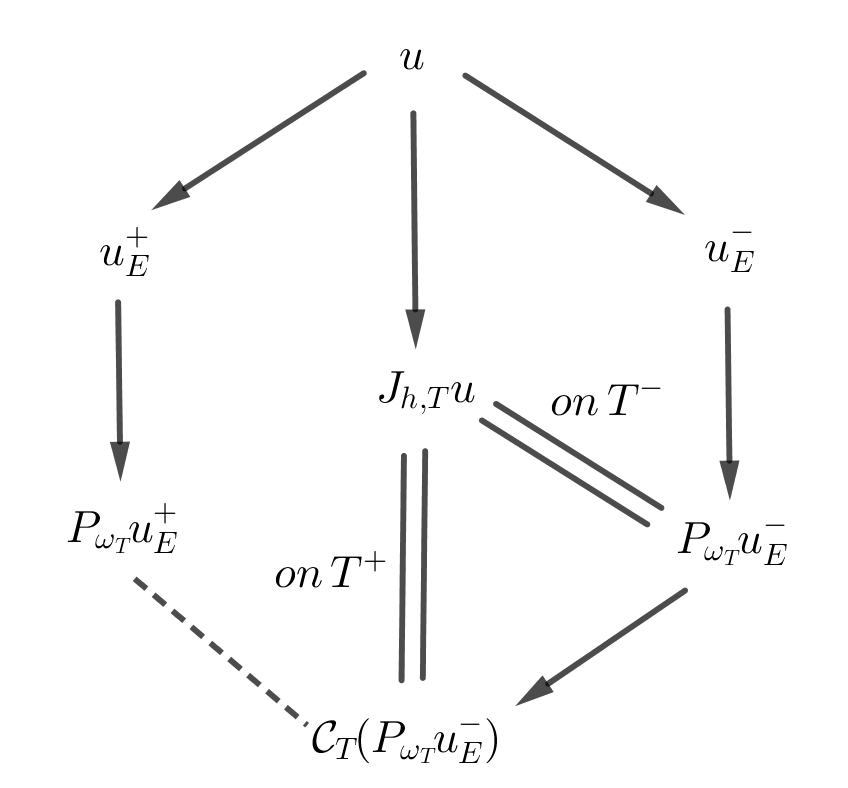}
  \caption{Diagram for analyzing approximation capabilities}
  \label{fig:diagram}
\end{minipage}
\end{figure}

Now, for every interface element $T$, we let $P_{\omega_T}$ be the standard projection operator from $H^2(\omega_T)$ to $\mathbb{Q}_1$ on $\omega_T$. Also, since functions in $S_h(T)$ are piecewise polynomials, we can use each of their polynomial components on the whole patch $\omega_T$ in the following discussion. To analyze the approximation capabilities of local IFE spaces $S_h(T)$, we consider the following ``interpolation" operator on the patch $\omega_T$ of $T$: $J_h~:~PH^2(\omega_T)\rightarrow S_h(T)$ such that
\begin{equation}
\label{Jh_def}
J_{h,T} u =
\left\{\begin{array}{cc}
J^-_{h,T} u := P_{\omega_T} u^-_E  ~~~~~~~~~ &\text{on} ~ T^- ,\\
J^+_{h,T} u := \mathcal{C}_T (P_{\omega_T} u^-_E) ~~~~~ &\text{on} ~ T^+,
\end{array}\right.
\end{equation}
in which $\mathcal{C}_T$ is the extension operator defined in \eqref{ext_opera}. By definition, $J_{h,T} u \in S_h(T)$, and our goal
is to show that $J_{h,T} u$ can approximate $u \in PH^2(\omega_T)$ optimally. The key idea for the analysis of the error in
$J_{h,T} u$ is illustrated by the diagram in Figure \ref{fig:diagram} in which each solid arrow indicates a well-understood relation and our main work is to estimate the difference between $P_{\omega_T}u^+_E$ and $\mathcal{C}_T(P_{\omega_T}u^-_E)$. We will use the following
functional to gauge this error:
\begin{equation}
\label{norm}
\vertiii{v}_{h,\omega_T} = \| v \|_{L^2(\tau_{\omega_T})} +  h \| \nabla v\cdot\bar{\mathbf{ n}} \|_{L^2(\tau_{\omega_T})} +  h^{3/2} | v |_{H^2(\omega_T)},~~\forall v \in H^2(\omega_T).
\end{equation}
We now show that $\vertiii{\cdot}_{h,\omega_T}$ is a norm equivalent to $\|\cdot\|_{L^2(\omega_T)}$ for functions
in $\mathbb{Q}_1$.
\begin{lemma}
\label{lem_norm_equi}
$\vertiii{\cdot}_{h,\omega_T}$ is a norm on $\mathbb{Q}_1$. Furthermore, there exist constants $c$ and $C$ such that
\begin{equation}
\label{norm_equi_eq0}
c\|p\|_{L^2(\omega_T)}  \leqslant  h^{1/2} \vertiii{p}_{h,\omega_T} \leqslant C \|p\|_{L^2(\omega_T)}, ~~~ \forall p\in \mathbb{Q}_1.
\end{equation}
\end{lemma}
\begin{proof}
Since $\|\cdot\|_{L^2(\omega_T)}$ is a norm, we only need to prove \eqref{norm_equi_eq0}. Fix a point $X_0\in\tau_{\omega_T}$, then for each $p\in\mathbb{Q}_1$,
\begin{equation}
\label{norm_equi_eq1}
p(X) = p(X_0) + \nabla p(X_0)\cdot(X-X_0) + (X-X_0)^TH_p(X-X_0), ~~~ \forall X\in \omega_T,
\end{equation}
where $H_p$ is the Hessian matrix of $p$ and it is a constant matrix. Then, we have
\begin{equation}
\begin{split}
\label{norm_equi_eq2}
\| p \|_{L^2(\omega_T)} &\leqslant C h^{3/2} ( \|p(X_0)\| + h \|\nabla p(X_0)\| + h^2 \|H_p\| ).
\end{split}
\end{equation}
By Lemma \ref{pt_tau}, we have $\|p(X_0)\| \leqslant  C h^{-1} \| p \|_{L^2(\tau_{\omega_T})}$. Besides, we note that $\|\nabla p(X_0)\|\leqslant \|\nabla p(X_0)\cdot\bar{\mathbf{  n}}\| + \| \nabla p(X_0)\cdot\bar{\mathbf{  t}}_1 \| + \| \nabla p(X_0)\cdot\bar{\mathbf{  t}}_2 \| $ where $\bar{\mathbf{ t}}_1$ and $\bar{\mathbf{ t}}_2$ are two orthogonal tangential vectors to $\tau_{\omega_T}$ at $X_0$. Lemma \ref{pt_tau} directly implies $\|\nabla p(X_0)\cdot\bar{\mathbf{  n}}\| \leqslant C h^{-1}\|\nabla p\cdot\bar{\mathbf{ n}} \|_{L^2(\tau_{\omega_T})}$. Since $p$ and $\nabla p\cdot\bar{\mathbf{  t}}_i=\partial_{\bar{\mathbf{ t}}_i}p$ can be considered as a two-variable polynomial and its derivatives on $\tau_{\omega_T}$, we then apply \eqref{tau_lambda_diam}, Lemma \ref{pt_tau} and the inverse inequality given by Lemma 3.3 in \cite{2016WangXiaoXu} to obtain
\begin{equation}
\label{norm_equi_eq3}
\| \nabla p(X_0)\cdot\bar{\mathbf{  t}}_i \|\leqslant C h^{-1} \| \partial_{\bar{\mathbf{ t}}_i} p \|_{L^2(\tau_{\omega_T})} \leqslant Ch^{-1} \frac{1}{\lambda_{\omega_T}} \| p \|_{L^2(\tau_{\omega_T})} \leqslant Ch^{-2} \| p \|_{L^2(\tau_{\omega_T})}.
\end{equation}
Furthermore, note that $H_p$ is a constant matrix, we have $ h^{2} \|H_p\| = Ch^{1/2} |p|_{H^2(\omega_T)}$. Putting these estimates into \eqref{norm_equi_eq2}, we have the first inequality in \eqref{norm_equi_eq0}. The second inequality in \eqref{norm_equi_eq0} directly follows from Lemma \ref{lem_trace_inequa_pyd} and the standard inverse inequality for polynomials.
\end{proof}
The following lemma is a preparation for the estimation of $\mathcal{C}_T(P_{\omega_T}u^-_E)-P_{\omega_T}u^+_E$.
\commentout{
Since the IFE functions determined by \eqref{ife_shape_fun_1} and \eqref{ife_shape_fun_2} only satisfy the flux jump condition at a point $F\in\tau_{\omega_T}$, we further need the following result to estimate $\mathcal{C}_T(P_{\omega_T}u^-_E)-P_{\omega_T}u^+_E$.
}
\begin{lemma}
\label{lem_IFE_tau_est}
There exists a constant $C$ such that
\begin{equation}
\label{IFE_tau_est_eq0}
\| \beta^+\nabla v^+_h\cdot\bar{\mathbf{ n}} - \beta^-\nabla v^-_h\cdot\bar{\mathbf{ n}} \|_{L^2(\tau_{\omega_T})} \leqslant
C (\beta^+ - \beta^-) h^{1/2} \, |v^s_h|_{H^2(\omega_T)}, ~s=\pm,~\forall v_h\in S_h(T), ~\forall T \in \mathcal{T}_h^i.
\end{equation}
\end{lemma}
\begin{proof}
Without loss of generality, we only prove the case $s=+$ in \eqref{IFE_tau_est_eq0}. The relation \eqref{thm_poly_rela_eq1} yields $\beta^+\nabla v^+_h\cdot\bar{\mathbf{ n}} - \beta^-\nabla v^-_h\cdot\bar{\mathbf{ n}} = (\beta^+ - \beta^-) (\nabla v^+_h\cdot\bar{\mathbf{ n}} - \nabla v^+_h(F) \cdot\bar{\mathbf{ n}} ) =: w$ which is a linear polynomial with $w(F)=0$ by \eqref{ife_shape_fun_2_2}. Then, by applying similar arguments as \eqref{norm_equi_eq1} and \eqref{norm_equi_eq2} with
$X_0=F\in \tau_{\omega_T}$ on $\tau_{\omega_T}$, we have
\begin{equation}
\label{IFE_tau_est_eq1}
\| w \|_{L^2(\tau_{\omega_T})} \leqslant C h^{2} \|\nabla w \|  \leqslant C (\beta^+ - \beta^-) h^{1/2} | v^+_h |_{H^2(\omega_T)},
\end{equation}
where we have also used that $\nabla w = (\beta^+ - \beta^-) \nabla(\nabla v^+_h\cdot\bar{\mathbf{ n}})$ is a constant vector.
\end{proof}

Now, we are ready to estimate $\mathcal{C}_T(P_{\omega_T}u^-_E)-P_{\omega_T}u^+_E$ indicated by the dashed line in the Diagram \ref{fig:diagram}.
\begin{lemma}
\label{lem_proj_ext_estimate}
Assume that the mesh $\mathcal{T}_h$ satisfies the conditions stated in Lemma \ref{lem_upm_tau}. Then
there exists a constant $C$ such that for every $u\in PH^2(\omega_T)$ the following holds:
\begin{equation}
\label{proj_ext_estimate_eq0}
 | \mathcal{C}_T(P_{\omega_T}u^-_E)-P_{\omega_T}u^+_E |_{H^k(\omega_T)} \leqslant  C h^{2-k} \sum_{s=\pm}( |u^s_E|_{H^1(\omega_T)} +  |u^s_E|_{H^2(\omega_T)} ), ~k=0,1,2, ~\forall T \in \mathcal{T}_h^i.
\end{equation}
\end{lemma}
\begin{proof}
Let $w=\mathcal{C}_T(P_{\omega_T}u^-_E)-P_{\omega_T}u^+_E$. Since $w \in \mathbb{Q}_1$, by Lemma \ref{lem_norm_equi}, we have
\begin{equation}
\label{proj_ext_estimate_eq1}
 Ch^{-1/2}\| w \|_{L^2(\omega_T)} \leqslant \vertiii{w}_{h,T} =  \| w \|_{L^2(\tau_{\omega_T})} +  h \| \nabla w\cdot\bar{\mathbf{ n}} \|_{L^2(\tau_{\omega_T})} +  h^{3/2} | w |_{H^2(\omega_T)} = I + II + III.
\end{equation}
For the term $I$, using the continuity condition on $\tau_{\omega_T}$, i.e., the first equation in \eqref{ife_shape_fun_2_1}, the triangular inequality, the trace inequality \eqref{trace_inequa_pyd_eq0}, and the estimate \eqref{upm_tau_eq01}, we have
\begin{align}
I & = \| P_{\omega_T} u^-_E -  P_{\omega_T}u^+_E \|_{L^2(\tau_{\omega_T})} \leqslant \sum_{s=\pm} \| P_{\omega_T} u^s_E -  u^s_E \|_{L^2(\tau_{\omega_T})} + \| u^+_E - u^-_E \|_{L^2(\tau_{\omega_T})} \nonumber \\
&\leqslant \sum_{s=\pm} ( h^{-1/2} \| P_{\omega_T} u^s_E -  u^s_E \|_{L^2(\omega_T)} + h^{1/2} \| P_{\omega_T} u^s_E -  u^s_E \|_{H^1(\omega_T)}) + \| u^+_E - u^-_E \|_{L^2(\tau_{\omega_T})} \label{proj_ext_estimate_eq2} \\
&  \leqslant C h^{3/2} \sum_{s = \pm} (| u^s_E |_{H^1(\omega_T)} + | u^s_E |_{H^2(\omega_T)}) . \nonumber
\end{align}
For the term $II$, firstly, by Lemma \ref{lem_IFE_tau_est} and the assumption $\beta^+\geqslant\beta^-$, we have
\begin{equation}
\label{proj_ext_estimate_eq3}
(\beta^+)^{-1} \| \beta^+ \nabla \mathcal{C}_T (P_{\omega_T} u^-_E)\cdot\bar{\mathbf{ n}} -  \beta^- \nabla P_{\omega_T} u^-_E \cdot\bar{\mathbf{ n}}\|_{L^2(\tau_{\omega_T})}   \leqslant  C h^{1/2} |P_{\omega_T} u^+_E|_{H^2(\omega_T)} \leqslant Ch^{1/2} |u^+_E|_{H^2(\omega_T)}.
\end{equation}
Secondly, using an argument similar to \eqref{proj_ext_estimate_eq2} with \eqref{upm_tau_eq02} and trace inequality \eqref{trace_inequa_pyd_eq0}, we obtain
\begin{align}
& (\beta^+)^{-1} \| \beta^- \nabla P_{\omega_T} u^-_E\cdot\bar{\mathbf{ n}} - \beta^+ \nabla P_{\omega_T} u^+_E \cdot\bar{\mathbf{ n}} \|_{L^2(\tau_{\omega_T})} \nonumber \\
 \leqslant  & (\beta^+)^{-1} \left( \sum_{s=\pm} \| \beta^s \nabla P_{\omega_T} u^s_E\cdot\bar{\mathbf{ n}} - \beta^s \nabla u^s_E \cdot\bar{\mathbf{ n}} \|_{L^2(\tau_{\omega_T})} + \| \beta^- \nabla u^-_E\cdot\bar{\mathbf{ n}} - \beta^+ \nabla u^+_E \cdot\bar{\mathbf{ n}} \|_{L^2(\tau_{\omega_T})} \right) \nonumber \\
 \leqslant & C h^{1/2} \sum_{s=\pm}   (| u^s_E |_{H^1(\omega_T)} +  | u^s_E |_{H^2(\omega_T)} ). \label{proj_ext_estimate_eq4}
\end{align}
Then, by triangular inequality together with \eqref{proj_ext_estimate_eq3} and \eqref{proj_ext_estimate_eq4}, we arrive at
\begin{align}
II \leqslant &h(\beta^+)^{-1} \left( \| \beta^+ \nabla \mathcal{C}_T (P_{\omega_T} u^-_E)\cdot\bar{\mathbf{ n}} -  \beta^- \nabla P_{\omega_T} u^-_E \cdot\bar{\mathbf{ n}}\|_{L^2(\tau_{\omega_T})}  + \| \beta^- \nabla P_{\omega_T} u^-_E\cdot\bar{\mathbf{ n}} - \beta^+ \nabla P_{\omega_T} u^+_E \cdot\bar{\mathbf{ n}} \|_{L^2(\tau_{\omega_T})} \right) \nonumber \\
\leqslant & C h^{3/2} \sum_{s=\pm}  ( | u^s_E |_{H^1(\omega_T)} +  | u^s_E |_{H^2(\omega_T)} ). \label{proj_ext_estimate_eq5}
\end{align}
For the term $III$, using the second condition in \eqref{ife_shape_fun_2_1} and noting that the second derivative of $\mathcal{C}_T (P_{\omega_T} u^-_E)\in\mathbb{Q}_1$ depends only on the coefficient of $xy+xz+yz$, we have
\begin{equation}
\label{proj_ext_estimate_eq6}
III =  h^{3/2} | P_{\omega_T} u^-_E -  P_{\omega_T}u^+_E |_{H^2(\omega_T)} \leqslant C  h^{3/2}( |u^-_E|_{H^2(\omega_T)} + |u^+_E|_{H^2(\omega_T)} ).
\end{equation}
Finally, putting \eqref{proj_ext_estimate_eq2}, \eqref{proj_ext_estimate_eq5}, and \eqref{proj_ext_estimate_eq6} into \eqref{proj_ext_estimate_eq1}, we obtain \eqref{proj_ext_estimate_eq0} for $\|\mathcal{C}_T(P_{\omega_T}u^-_E)-P_{\omega_T}u^+_E\|_{L^2(\omega_T)}$.
The results for $|\mathcal{C}_T(P_{\omega_T}u^-_E)-P_{\omega_T}u^+_E|_{H^k(\omega_T)}$, $k=1, 2$, follow from the standard inverse inequality for polynomials.
\end{proof}
Then, we have the approximation capabilities of the proposed local IFE space \eqref{loc_IFE_space} on interface elements.
\begin{thm}
\label{lem_Jh_est}
Assume that the mesh $\mathcal{T}_h$ satisfies the conditions stated in Lemma \ref{lem_upm_tau}. Then
there exists a constant $C$ such that for every $u\in PH^2(\omega_T)$ the following holds:
\begin{equation}
\label{Jh_est_eq0}
| J^t_{h,T} u  - u^t_E |_{H^k(\omega_T)} \leqslant C h^{2-k} \sum_{s=\pm}( |u^s_E|_{H^1(\omega_T)} +  |u^s_E|_{H^2(\omega_T)} ),
~~ k=0,1,2, ~~ t=\pm, \forall T \in \mathcal{T}_h^i.
\end{equation}
\end{thm}
\begin{proof}
The argument is outlined by the diagram in Figure \ref{fig:diagram}. For the case when $t = +$,
we note that $J^-_{h,T} u  - u^-_E=P_{\omega_T}u^-_E - u^-_E$. Hence,
estimate \eqref{Jh_est_eq0} follows directly from the approximation property of the standard projection operator $P_{\omega_T}$ to $\mathbb{Q}_1$. For the case when $t = +$, we note $J^+_{h,T} u  - u^+_E = (\mathcal{C}_T(P_{\omega_T}u^-_E) - P_{\omega_T}u^+_E) + (P_{\omega_T}u^+_E - u^+_E)$. Thus, \eqref{Jh_est_eq0} follows from Lemma \ref{lem_proj_ext_estimate} and the approximation property of the projection operator $P_{\omega_T}$.
\end{proof}
\commentout{
A direct consequence of Theorem \ref{lem_Jh_est} is
\begin{equation}
\label{Jh_interp_est}
| J_{h,T} u  - u |_{H^k(\omega_T)} \leqslant C h^{2-k} \sum_{s=\pm}( |u^s_E|_{H^1(\omega_T)} +  |u^s_E|_{H^2(\omega_T)} ),
~~ k=0,1,2, ~~ t=\pm, \forall T \in \mathcal{T}_h^i.
\end{equation}
}



The result in Theorem \ref{lem_Jh_est} can be also used to analyze the Lagrange type interpolation, i.e., on every interface
element $T$, we define $I_{h,T}~:~PH^2(\omega_T)\rightarrow S_h(T)$ such that
\begin{equation}
\label{lagrange_interp}
I_{h,T} u = \sum_{i\in\mathcal{I}} u(A_i)\phi_{i,T},
\end{equation}
where $\phi_{i,T}$ are the IFE shape functions determined by \eqref{IFE_shape_fun_delta} and $A_i$, $i\in\mathcal{I}$ are vertices of $T$. Again, functions in $S_h(T)$ are understood as piecewise functions whose component polynomials can be used on the whole patch $\omega_T$. First, we show that the IFE shape functions have bounds similar to those of the $\mathbb{Q}_1$ finite element shape functions on $T$. We also recall the assumption that $\beta^+\geqslant\beta^-$.
\begin{thm}[Bounds of IFE shape functions]
\label{bounds_IFEshapeFun}
Let $\mathcal{T}_h$ be a mesh satisfying the \textbf{Assumptions} \textbf{(H1)}-\textbf{(H3)}. Then, there exists a constant $C$ independent of the interface location, mesh size $h$, and coefficients $\beta^{\pm}$ such that
\begin{equation}
\label{bounds_IFEshapeFun_eq01}
| \phi_{i,T} |_{k,\infty,\omega_T^+} \leqslant C h^{-k} ~~~ \text{and} ~~~ | \phi_{i,T} |_{k,\infty,\omega_T^-} \leqslant C\frac{\beta^+}{\beta^-} h^{-k}, ~~~ k=0,1, 2 ~~i\in\mathcal{I}, ~~\forall T \in \mathcal{T}_h^i.
\end{equation}
\end{thm}
\begin{proof}
Since $|\psi_{i,T}|_{k,\infty,\omega_T}\leqslant Ch^{-k}$, $k=0,1,2$, we only need to estimate the coefficients $\mathbf{ c}$ and $c_0$ in \eqref{ife_shape_fun_4}. Let $\mathbf{ e}=(e_i)_{i\in\mathcal{I}}^T$ be one of the unit vectors in $\mathbb{R}^8$, i.e., $(1,0,\cdots,0)$, $(0,1,\cdots,0)$, $\cdots$, $(0,0,\cdots,1)$, and $\mathbf{ e}^-=(e_i)_{i\in\mathcal{I}^-}$. Then, under the notations of \eqref{c_c0_formula} with $\mathbf{ v}=\mathbf{ e}$, the fact $| \Xi |\leqslant Ch^{-1}$ and Lemma \ref{lemma_gammadelta_01} imply
\begin{equation}
\label{bounds_IFEshapeFun_eq2}
\| \mathbf{ c} \|_{\infty} \leqslant \| \mathbf{ e}^- \|_{\infty} +  \frac{ |\mu| |\Xi| \| \bfdelta \|_{\infty} }{ 1 + \mu \bfgamma^T \bfdelta }
\leqslant C + \frac{C \mu \bfgamma^T\bfdelta }{1 + \mu \bfgamma^T\bfdelta} \leqslant C,
\end{equation}
where we have used $\mu\geqslant0$ because $\beta^+\geqslant\beta^-$. Similarly, by \eqref{c_c0_formula} and $\| L \|_{k,\infty,\omega_T}\leqslant Ch^{1-k}$, we have
\begin{equation}
\label{bounds_IFEshapeFun_eq3}
\| c_0L \|_{k,\infty,\omega_T^{-}} \leqslant \frac{Ch^{-1} \mu \|L\|_{k,\infty,\omega_T^-} }{1 + \mu \bfgamma^T\bfdelta} \leqslant C\mu h^{-k} \leqslant C \frac{\beta^+}{\beta^-} h^{-k} , ~~~ k=0,1,2.
\end{equation}
where \eqref{bounds_IFEshapeFun_eq3} is trivial for $k=2$, since $L$ is linear. Putting \eqref{bounds_IFEshapeFun_eq2} and \eqref{bounds_IFEshapeFun_eq3} into \eqref{ife_shape_fun_4}, we have \eqref{bounds_IFEshapeFun_eq01}.
\end{proof}

\begin{thm}
\label{thm_Ih_ext}
Assume that the mesh $\mathcal{T}_h$ satisfies the conditions stated in Lemma \ref{lem_upm_tau}.
Then there exists a constant $C$ such that for every $u\in PH^2(\omega_T)$ the following holds:
\begin{equation}
\label{Ih_est_eq0}
| I_{h,T} u  - u |_{H^k(\omega_T)} \leqslant C \frac{\beta^+}{\beta^-} h^{2-k} \sum_{s=\pm} ( | u^s_E |_{H^1(\omega_T)} +  | u^s_E |_{H^2(\omega_T)} ), ~~k=0,1,2, ~~\forall T \in \mathcal{T}_h^i.
\end{equation}
\end{thm}
\begin{proof}
Let $w=J_{h,T}u - u$ and $w^s=J^s_{h,T}u - u^s_E$, $s=\pm$. Note that each $w^s$ can be considered as a function on the whole patch $\omega_T$. Using the fact $I_{h,T}J_{h,T}u=J_{h,T}u$ and the triangular inequality, we have
\begin{equation}
\label{Ih_ext_eq1}
| I_{h,T}u - u |_{H^{k}(\omega_T)} \leqslant | I_{h,T}u - J_{h,T}u |_{H^{k}(\omega_T)} + | J_{h,T}u - u |_{H^{k}(\omega_T)} \leqslant |I_{h,T}w|_{H^k(\omega_T)} + |w|_{H^k(\omega_T)}.
\end{equation}
Then, by Theorem \ref{bounds_IFEshapeFun}, the Sobolev imbedding Theorem and the scaling argument, we have
\begin{equation}
\begin{split}
\label{Ih_ext_eq2}
| I_{h,T}w|_{H^k(\omega_T)} \leqslant C \frac{\beta^+}{\beta^-}  \sum_{j=0}^2 h^{j-k}( |w^-|_{H^j(\omega_T)} + |w^+|_{H^j(\omega_T)} ), ~~~ k=0,1,2. \\
\end{split}
\end{equation}
Therefore, \eqref{Ih_est_eq0} follows from \eqref{Ih_ext_eq1} and \eqref{Ih_ext_eq2} together with the approximation results \eqref{Jh_est_eq0}.
\end{proof}

Since each interface element $T$ is a subset of its patch $\omega_T$, estimates established in Theorem \ref{lem_Jh_est} and
Theorem \ref{thm_Ih_ext} imply that $J_{h, T}u$ and $I_{h, T}u$ can approximate $u \in PH^2(\omega_T)$ optimally
with respect to the underlying polynomials space $\mathbb{Q}_1$. As usual, these local optimal approximation capability further imply
the optimal approximation capability of the global IFE space $S_h(\Omega)$ defined on the whole $\Omega$. We can see this from
the global IFE interpolation operator $I_h : PH^2(\Omega)\rightarrow S_h(\Omega)$ defined piecewisely such that $I_hu|_T=I_{h,T}u$, in which $I_{h,T}$ is given by \eqref{lagrange_interp} on $T\in\mathcal{T}^i_h$ and the standard Lagrange interpolation on $T\in\mathcal{T}^n_h$. Then, applying the standard estimation results for the Lagrange interpolation \cite{2008BrennerScott} on each non-interface element, summing these estimates and \eqref{Ih_est_eq0} over all the elements, and using the finite overlapping property of the patches as well as the extension boundedness \eqref{sobolev_ext}, we have
\begin{equation}
\label{glob_Ih_est}
h^k | I_h u - u |_{H^k(\Omega)} \leqslant C \frac{\beta^+}{\beta^-}h^2 |u|_{H^2(\Omega)}, ~~~ k=0,1,2,
\end{equation}
which means that $I_h u$ is an optimal approximation of $u \in PH^2(\Omega)$.

\commentout{
Since each interface element $T$ is a subset of its patch $\omega_T$, estimates established in Theorem \ref{lem_Jh_est} and
Theorem \ref{thm_Ih_ext} imply that $J_{h, T}u$ and $I_{h, T}u$ can approximate $u \in PH^2(\omega_T)$ optimally
with respect to the underlying polynomials space $\mathbb{Q}_h$. As usual, these local optimal approximation capability further imply
the optimal approximation capability of the IFE space $S_h(\Omega)$ global defined on the whole $\Omega$. {\color{red} By the discussion above, we can demonstrate this for two different interpolations on $S_h(\Omega)$ defined according to \eqref{glob_IFE_space_DG} or \eqref{glob_IFE_space}. For the global space \eqref{glob_IFE_space_DG}, we take the global IFE interpolation operator $J_h : PH^2(\Omega)\rightarrow S_h(\Omega)$ defined elementwisely such that $J_hu|_T=J_{h,T}u$, in which $J_{h,T}$ is given by \eqref{Jh_def} on $T\in\mathcal{T}^i_h$ and the standard Lagrange interpolation on $T\in\mathcal{T}^n_h$. Similarly, for \eqref{loc_IFE_space}, we define $I_h : PH^2(\Omega)\rightarrow S_h(\Omega)$ such that $I_hu|_T=I_{h,T}u$, in which $I_{h,T}$ is given by \eqref{lagrange_interp} on $T\in\mathcal{T}^i_h$ and the standard Lagrange interpolation on $T\in\mathcal{T}^n_h$. Then, applying the standard estimation results for the Lagrange interpolation \cite{2008BrennerScott} on each non-interface element, summing these estimates and \eqref{Jh_interp_est} or \eqref{Ih_est_eq0} over all the elements, and using the finite overlapping property of the patches as well as the extension boundedness \eqref{sobolev_ext}, we have
\begin{equation}
\label{glob_Jh_est}
h^k | J_h u - u |_{H^k(\Omega)} \leqslant C h^2 \| u \|_{H^2(\Omega)}, ~~~ k=0,1,2.
\end{equation}
\begin{equation}
\label{glob_Ih_est}
h^k | I_h u - u |_{H^k(\Omega)} \leqslant C \frac{\beta^+}{\beta^-}h^2 \| u \|_{H^2(\Omega)}, ~~~ k=0,1,2.
\end{equation}
This means that both the types of global IFE space have optimal approximation capabilities of $u \in PH^2(\Omega)$. We note that the global IFE space defined in a DG manner \eqref{glob_IFE_space_DG} has the optimal approximation capabilities independent of the discontinuous coefficients $\beta^{\pm}$ which may be advantageous for the large jump case, while the global IFE space \eqref{glob_IFE_space} is isomorphic to the standard trilinear finite element space defined on the same mesh, in terms of the number and the location of their global degrees of freedom, independent of the interface shape and location, which may be useful in solving moving interface problems. }
}

\commentout{
\begin{rem}
The special interpolation operator $J_{h,T}$ has the feature that the related approximation capabilities are independent of the coefficient $\beta^{\pm}$. We think this feature together with a global IFE space discontinuous around the interface \cite{2015GuzmanSanchezSarkisP1} may be useful in the analysis of a discontinuous Galerkin scheme for solving high-contrast interface problems. Here, the standard Lagrange interpolation $I_{h,T}$ and the related global interpolation $I_h$ enable us to analyze the approximation capabilities of the proposed global IFE space $S_h(\Omega)$ in \eqref{glob_IFE_space} with the degrees of freedom independent of the interface. As described before, we believe this space together with the PPIFE scheme \eqref{ppife_1}-\eqref{ppife} may have advantages in solving moving interface problems. The analysis of these schemes will be discussed in a forthcoming article.
\end{rem}
}





\section{Numerical Examples}
\label{sec:num_examp}

In this section, we present a group of numerical examples to demonstrate features of the proposed IFE space. Since we have known that the proposed IFE space has an optimal approximation capability, we naturally expect that it can be used to solve the
interface problem \eqref{model}. For this purpose, we consider the extension of the partially penalized immersed finite element (PPIFE)
method \cite{2015LinLinZhang} for 2-D interface problems to the proposed 3-D IFE space. To describe this method,
\commentout{
we let
\begin{equation}
\begin{split}
\label{underline_spa}
V_h(\Omega) = &\left\{ v\in L^2(\Omega)~:~ v\in H^2(T)~\forall T\in\mathcal{T}^n_h, ~ v\in H^2(T^{\pm})~\forall T\in\mathcal{T}^i_h; \right. \\
& \hspace{0.2in} \left. \text{and} ~ v ~ \text{is continuous at each} ~ X\in\mathcal{N}_h ~ \text{and each} ~ e\in\mathcal{E}^n_h \right\},
\end{split}
\end{equation}
}
we use the IFE space $S_h(\Omega)$ defined in \eqref{glob_IFE_space} and consider a bilinear form $a:S_h(\Omega)\times S_h(\Omega)\rightarrow\mathbb{R}$ such that
\begin{equation}
\begin{split}
\label{ppife_1}
a_h(u,v) =& \sum_{T\in\mathcal{T}_h} \int_T \beta \nabla u \cdot \nabla v dX \\
 - & \sum_{F\in\mathring{\mathcal{F}}_h} \int_F \{ \beta \nabla u\cdot \mathbf{ n} \} [v] ds  + \epsilon \sum_{F\in\mathring{\mathcal{F}}_h} \int_F \{ \beta \nabla v\cdot \mathbf{ n} \} [u] ds + \sum_{F\in\mathring{\mathcal{F}}_h} \frac{\sigma^0}{|F|} \int_F [u]\,[v] ds,\\
\end{split}
\end{equation}
and a linear form $L: S_h(\Omega)\rightarrow \mathbb{R}$ such that 
\begin{equation}
\label{ppife_2}
L(v) = \int_{\Omega} f v dX. 
\end{equation}
Then, the PPIFE method for solving the interface problem \eqref{model} is to find $u_h\in S_h(\Omega)$ such that
\begin{equation}
\label{ppife}
a_h(u_h,v_h) = L(v_h), ~~~ \forall v_h\in S_{h, 0}(\Omega),
\end{equation}
where $S_{h, 0}(\Omega)$ is the subspace of $S_h(\Omega)$ formed by IFE functions with zero trace on $\partial \Omega$, and we tacitly
assume that the interface surface $\Gamma$ does not tough $\partial \Omega$. The terms on interface faces in \eqref{ppife_1} are similar to interior penalties in DG methods \cite{1982Arnold,1976DouglasDupont}; hence, we call this method described by \eqref{ppife_1}-\eqref{ppife} symmetric, non-symmetric and incomplete PPIFE (SPPIFE, NPPIFE and IPPIFE) methods when $\epsilon=-1,1$ and $0$, respectively.

Numerical results to be reported are generated by applying the PPIFE method to three interface problems posed in the domain
$\Omega = (-1, 1)^3$ whose interface surfaces have three representative geometries as shown in Figures \ref{fig:sphere}-\ref{fig:torus_sphere}, respectively. We only present numerical results of the SPPIFE method because our extensive numerical experiments suggest that the NPPIFE and IPPIFE schemes behave similarly.
In particular, we choose the stabilization parameter in \eqref{ppife_1} as $\sigma_0=10\max\{\beta^-, \beta^+\}$, all the data are generated on a sequence of meshes characterized by the mesh size $h$, and the convergence rates are estimated by numerical results on two consecutive meshes.

\begin{figure}[h]
\centering
\begin{minipage}{.3\textwidth}
  \centering
  \includegraphics[width=0.8\textwidth]{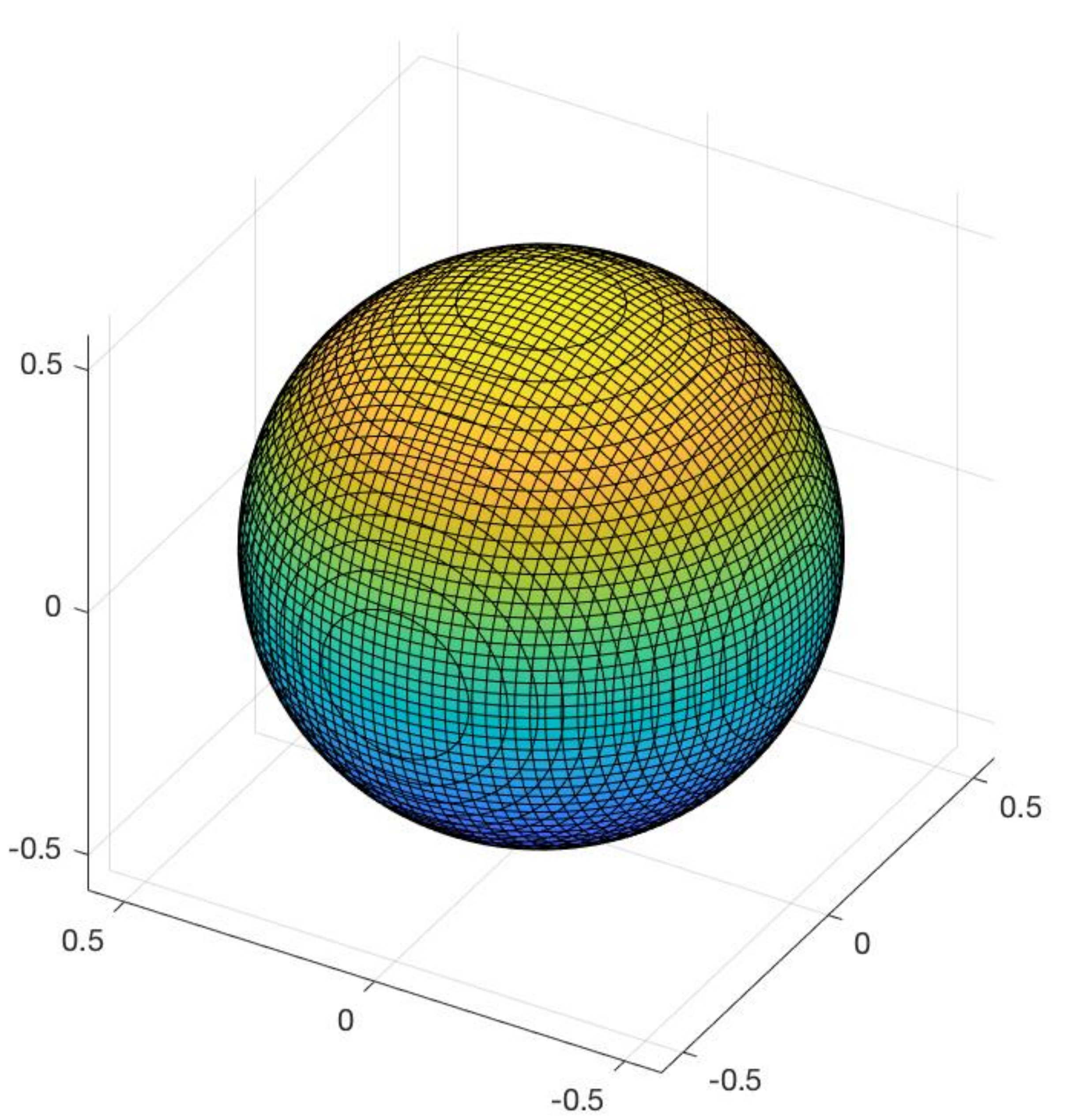}
  \caption{Sphere}
  \label{fig:sphere}
\end{minipage}
\begin{minipage}{.3\textwidth}
  \centering
  \includegraphics[width=0.85\textwidth]{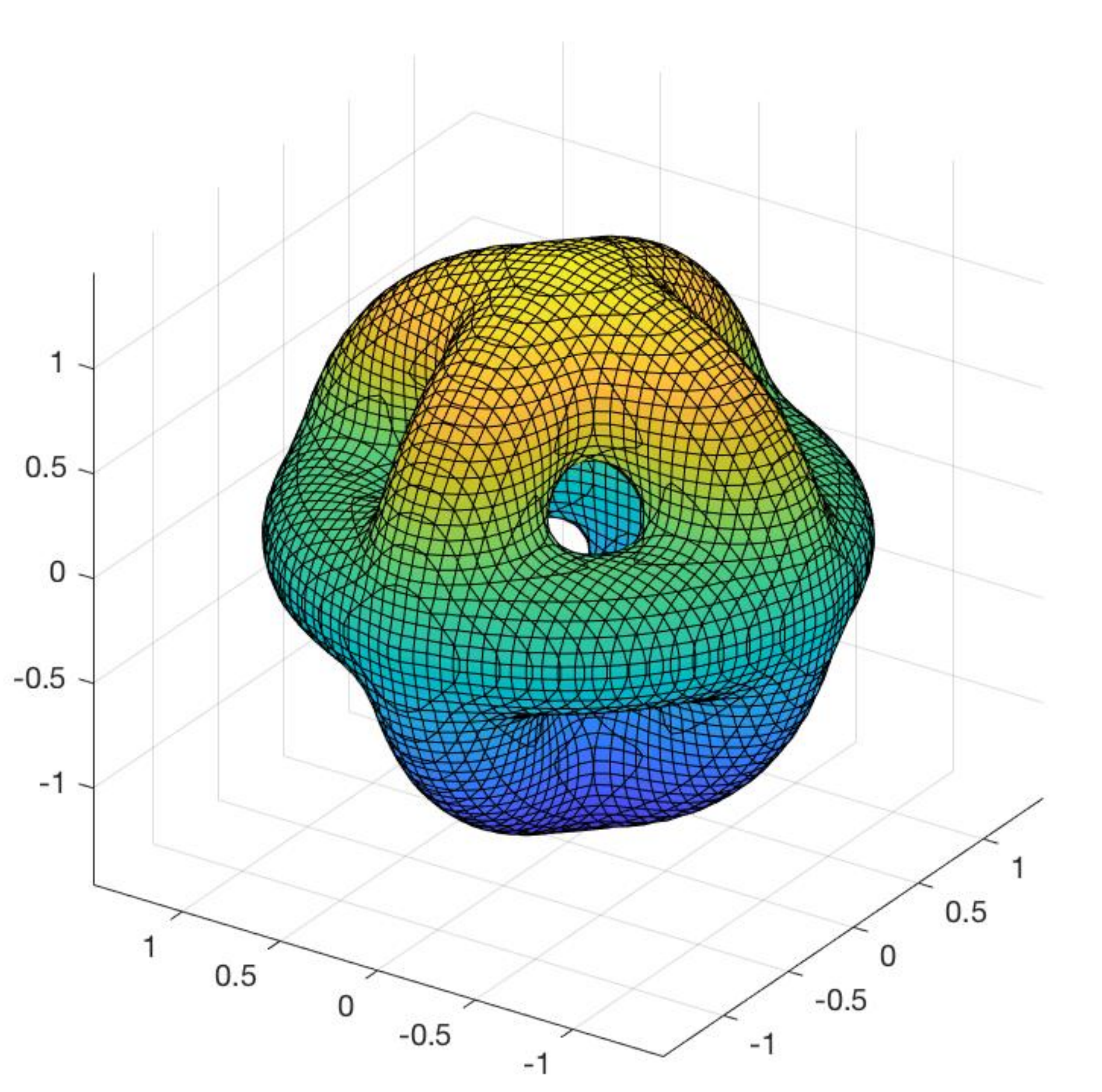}
  \caption{Orthocircle}
  \label{fig:torus3}
\end{minipage}
\begin{minipage}{.3\textwidth}
  \centering
  \includegraphics[width=0.7\textwidth]{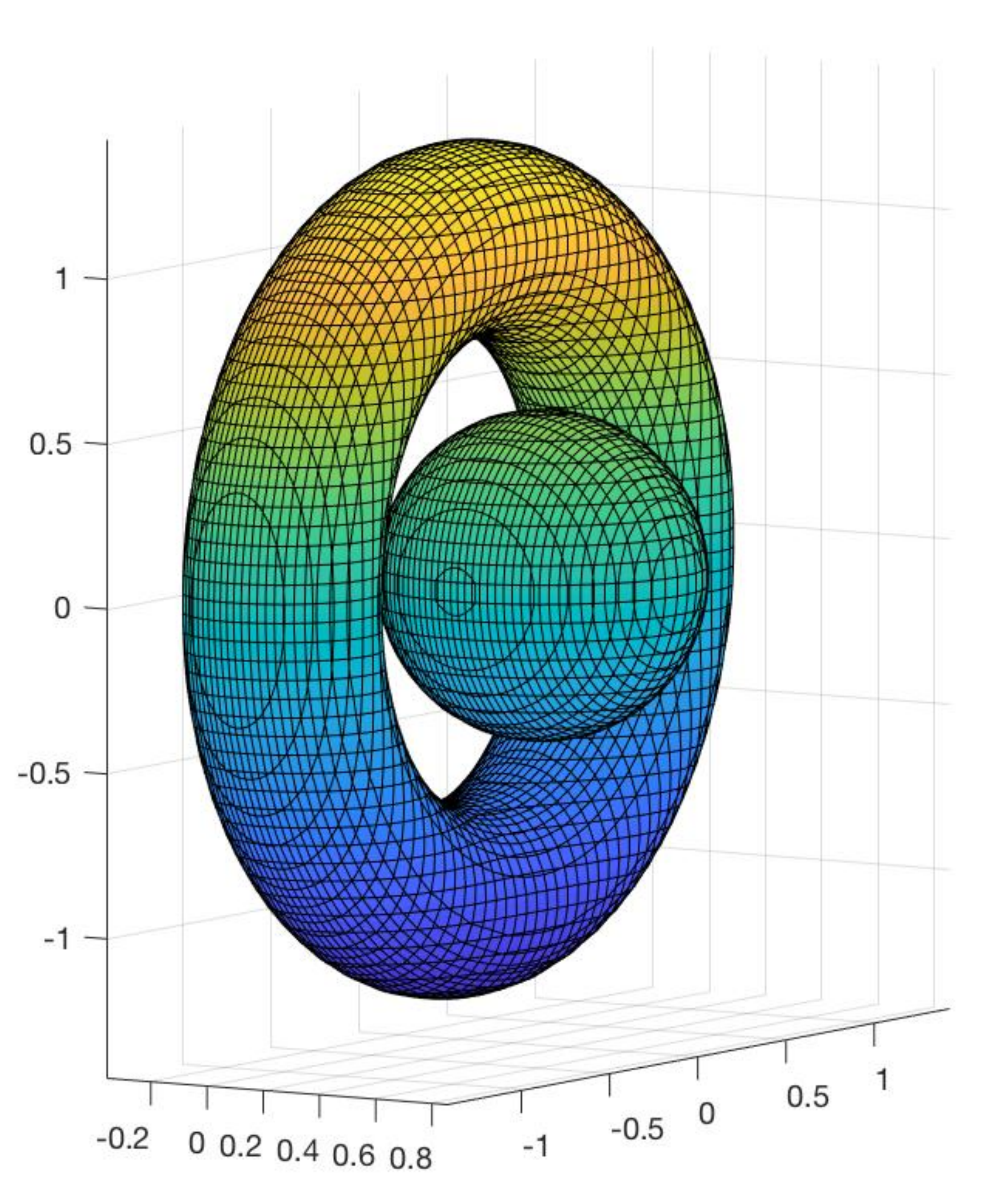}
  \caption{A torus and a sphere}
  \label{fig:torus_sphere}
\end{minipage}
\end{figure}

\commentout{
One of the key components in the implementation is the numerical quadrature for discontinuous IFE functions on interface elements. In generating the presented numerical results, we first employ a local triangularization of the cubic interface element according to the interface location. In this procedure, each tetrahedral subelement with a surface face will be approximated by the related linear tetrahedron. Then, we perform the standard quadrature rule on each of these tetrahedron, and sum the results over all the tetrahedral subelements to obtain the quadrature for discontinuous functions on interface elements. Alternatively, other special quadrature rules for discontinuous functions can be also applied such as the one based on the local smoothness \cite{2002Tornberg}. In this type of rules, the quadrature points are independent of the interface location which makes it easy for implementation, but in general much more points are needed. To avoid redundancy, we will not present the numerical results generated by this special quadrature rule since they behave similar to the one based on the local triangularization.
}

\noindent
{\bf Spherical Interface}: In the first example, the interface problem \eqref{model} has a simple spherical interface surface as shown in Figure \ref{fig:sphere} defined by a level-set: $\{X\in \mathbb{R}^3: w(X)=0\}$ with $w(X):=\sqrt{x_1^2 + x_2^2 + x_3^2} - r_0$ and $r_0=\pi/6$ such that $\Omega^-=\{X \in \Omega :  w(X)<0 \}$ and $\Omega^+ = \Omega - \Omega^-$.
We choose $f$ and $g$ in the interface problem \eqref{model} so that its exact solution is such that
$u=u^-(X) = r_0^6\left(w(x_1,x_2,x_3)/r_0+1\right)^{5}/\beta^-$ on $\Omega^-$ and $u=u^+(X) = r_0^6\left( w(x_1,x_2,x_3)/r_0+1\right)^{5}/\beta^+ + r^6_0\left(1/\beta^- - 1/\beta^+\right)$ on $\Omega^+$ with $\beta^-=1$, $\beta^+ =100$.
The numerical results for the Lagrange interpolation and SPPIFE solutions are presented in Tables \ref{table:examp_ell_interp_1_100} and \ref{table:examp_ell_ppife_1_100}, respectively. These data clearly demonstrate the optimal convergence for both the Lagrange interpolation and SPPIFE solution in $L^{2}$ and $H^1$ norms. We also note that the convergence rate in $L^{\infty}$ norm is close to optimal.

\begin{table}[H]
\begin{center}
\scriptsize{
\begin{tabular}{|c |c c|c c|c c|}
\hline
$h$   & $\| u - I_hu\|_{L^{\infty}(\Omega)}$ & rate  & $\| u - I_hu\|_{L^2(\Omega)}$ & rate   & $| u - I_hu|_{ PH^1(\Omega)}$ & rate   \\ \hline
1/20  &     1.8969e-03          &                   &   8.6758e-04          &                   &   1.9092e-02   &                  \\ \hline
1/30   &    9.3633e-04          &    1.7412    &   4.0365e-04          &   1.8871     &   1.3086e-02   &    0.9316    \\ \hline
1/40   &    5.8322e-04          &    1.6456    &   2.3393e-04          &   1.8963     &   1.0070e-02   &    0.9106    \\ \hline
1/50  &     3.9574e-04          &    1.7379    &   1.5217e-04          &   1.9270     &   8.1492e-03   &    0.9486     \\ \hline
1/60  &     2.7281e-04          &    2.0403    &   1.0675e-04          &   1.9444     &   6.8252e-03   &    0.9724     \\ \hline
1/70  &     2.0707e-04          &    1.7885    &   7.9031e-05          &   1.9504     &   5.8933e-03   &    0.9524     \\ \hline
1/80  &     1.5980e-04          &    1.9407    &   6.0869e-05          &   1.9554     &   5.1697e-03   &    0.9810     \\ \hline
1/90  &     1.2875e-04          &    1.8343    &   4.8356e-05          &   1.9540     &   4.6104e-03   &    0.9722     \\ \hline
1/100  &   1.0540e-04          &    1.8997    &   3.9327e-05          &   1.9616     &   4.1604e-03   &    0.9748      \\ \hline
\end{tabular}
}
\end{center}
\caption{Errors and rates of the Lagrange interpolation for $\beta^-=1$, $\beta^+=100$.}
\label{table:examp_ell_interp_1_100}
\end{table}

\begin{table}[H]
\begin{center}
\scriptsize{
\begin{tabular}{|c |c c|c c|c c|}
\hline
$h$   & $\| u - u_h\|_{L^{\infty}(\Omega)}$ & rate  & $\| u - u_h\|_{L^2(\Omega)}$ & rate   & $| u - u_h|_{ PH^1(\Omega)}$ & rate   \\ \hline
1/20  &     2.2425e-03          &                   &   1.1204e-03          &                   &    1.9938e-02    &                   \\ \hline
1/30   &    1.2337e-03          &    1.4738    &   5.5513e-04          &    1.7319    &    1.3609e-02    &    0.9420     \\ \hline
1/40   &    6.9686e-04          &    1.9855    &   3.0324e-04          &    2.1019    &    1.0328e-02    &    0.9589     \\ \hline
1/50  &     4.4909e-04          &    1.9689    &   1.9187e-04          &    2.0513    &    8.3026e-03    &    0.9781      \\ \hline
1/60  &     3.2664e-04          &    1.7461    &   1.3468e-04          &    1.9411    &    6.9389e-03    &    0.9842      \\ \hline
1/70  &     2.4643e-04          &    1.8280    &   9.7072e-05          &    2.1241    &    5.9748e-03    &    0.9704      \\ \hline
1/80  &     1.8649e-04          &    2.0871    &   7.5467e-05          &    1.8854    &    5.2362e-03    &    0.9883      \\ \hline
1/90  &     1.5431e-04          &    1.6079    &   5.9513e-05          &    2.0165    &    4.6595e-03    &    0.9906      \\ \hline
1/100  &   1.2467e-04          &    2.0243    &   4.8264e-05          &    1.9884    &    4.2026e-03    &    0.9795       \\ \hline
\end{tabular}
}
\end{center}
\caption{Errors and rates of the PPIFE solutions for $\beta^-=1$, $\beta^+=100$.}
\label{table:examp_ell_ppife_1_100}
\end{table}

By the analysis in Section \ref{sec:geometry} and Section \ref{sec:approximation}, we know that the rules we proposed in Section \ref{sec:geometry} to form the plane $\tau(T)$ for the construction of IFE functions in each interface element $T$ ensure the maximum angle property which further ensure the optimal approximation capability of the IFE space. On the other hand, our extensive numerical experiments indicate that the proposed trilinear IFE space will not possess the expected approximation capability in general if we do not follow these rules to form the plane $\tau(T)$, and we present two typical groups of data in Tables \ref{table:examp_ell_ppife_1_100_wrong} and \ref{table:examp_ell_ppife_1_100_right} for the corroboration of this observation. In these tables, we employ the following quantities to assess errors in the IFE interpolation of
the function $u$ described above:
\begin{equation}
\label{quantities}
\eta^{\infty}_h = \max_{T\in\mathcal{T}^i_h} \| u - I_h u \|_{L^{\infty}(T)}, ~~~~  \eta^{0}_h = \max_{T\in\mathcal{T}^i_h} \frac{ \| u - I_h u \|_{L^{2}(T)} }{\| u \|_{PH^2(T)}}, ~~~~ \eta^{1}_h = \max_{T\in\mathcal{T}^i_h} \frac{ | u - I_h u |_{H^{1}(T)} }{\| u \|_{PH^2(T)}}.
\end{equation}
Data in Table \ref{table:examp_ell_ppife_1_100_wrong} demonstrate the behavior of IFE functions constructed
with a ``wrong" approximation plane $\tau(T)$ determined by three interface points $D_{j_i}, i = 1, 2, 3$ randomly chosen in
each interface element $T$ so that the rules proposed for $\tau(T)$ are not always obeyed. The data in this table indicate that
these IFE functions do not even show any convergence. In contrast, data in Table \ref{table:examp_ell_ppife_1_100_right} demonstrate
the expected convergence of the IFE functions constructed with $\tau(T)$ formed with the proposed rules. From both the analysis and numerical experiments we can infer that the proposed rules to form the plane $\tau(T)$ for ensuring the maximum angle property is important.

\begin{table}[H]
\begin{center}
\scriptsize{
\begin{tabular}{|c |c c|c c|c c|}
\hline
$h$   & $\eta^{\infty}_h$      & rate             & $\eta^{0}_h$          & rate            & $\eta^{1}_h$     & rate   \\ \hline
1/20  &     1.8665e-03          &                    &    1.1782e-03         &                     &    4.0787e-02   &                  \\ \hline
1/30   &    5.0318e-03          &    -2.4459     &    1.7692e-03         &   -1.0027     &   1.2917e-01    &    -2.8431  \\ \hline
1/40   &    1.3626e-03          &      4.5411     &    5.9872e-04         &    3.7663     &    9.4378e-02   &    1.0909  \\ \hline
1/50  &     1.1089e-03          &      0.9231     &    3.7605e-04         &    2.0842     &    5.8381e-02   &   2.1525   \\ \hline
1/60  &     9.6156e-04          &     0.7821     &    3.3370e-04         &     0.6553    &    4.1506e-02   &    1.8712   \\ \hline
1/70  &     2.4471e-04          &      8.8777     &    1.1215e-04         &    7.0737     &    1.6304e-02   &   6.0618   \\ \hline
1/80  &     3.4541e-04          &     -2.5812     &    1.3698e-04         &  -1.4977     &    3.7622e-02   &    -6.2621   \\ \hline
1/90  &     3.9607e-04          &    -1.1620     &    1.3758e-04         &   -0.0371     &   5.8061e-02    &    -3.6840   \\ \hline
1/100  &   2.9810e-03          &    19.1571    &    9.7791e-04         &    -18.6146   &   2.2283e-01    &    -12.7647     \\ \hline
\end{tabular}
}
\end{center}
\caption{Errors and rates of $\eta^{\infty}_h$, $\eta^{0}_h$ and $\eta^{1}_h$ for $\beta^-=1$, $\beta^+=100$ by an inappropriate choice of $\tau$.}
\label{table:examp_ell_ppife_1_100_wrong}
\end{table}

\begin{table}[H]
\begin{center}
\scriptsize{
\begin{tabular}{|c |c c|c c|c c|}
\hline
$h$   & $\eta^{\infty}_h$      & rate             & $\eta^{0}_h$          & rate            & $\eta^{1}_h$     & rate   \\ \hline
1/20  &     1.8038e-03          &                   &   1.1782e-03        &                  &   4.0740e-02  &                  \\ \hline
1/30   &    9.3633e-04          &    1.6171     &   5.6690e-04         &  1.8043   & 2.8124e-02   &     0.9139  \\ \hline
1/40   &    5.8322e-04          &    1.6456     &  3.2608e-04         &   1.9224  &  2.2490e-02   &     0.7772  \\ \hline
1/50  &     3.7932e-04          &    1.9279     &  2.0603e-04         &   2.0575  &  1.8043e-02   &     0.9874    \\ \hline
1/60  &     2.6949e-04          &    1.8750    &   1.4117e-04        &     2.0736  &   1.4849e-02  &    1.0684    \\ \hline
1/70  &     2.0669e-04          &    1.7211      &  1.0552e-04         &   1.8884   &  1.3165e-02   &    0.7810    \\ \hline
1/80  &     1.5914e-04          &    1.9578     &  8.1105e-05         &    1.9705   &   1.1199e-02  &    1.2115     \\ \hline
1/90  &     1.2674e-04          &    1.9329     &   6.3621e-05         &  2.0615   & 1.0136e-02   &     0.8463      \\ \hline
1/100  &   1.0540e-04          &    1.7502    &   5.2089e-05        &    1.8982   &  9.3055e-03   &    0.8114      \\ \hline
\end{tabular}
}
\end{center}
\caption{Errors and rates of $\eta^{\infty}_h$, $\eta^{0}_h$ and $\eta^{1}_h$ for $\beta^-=1$, $\beta^+=100$ by the proposed $\tau$.}
\label{table:examp_ell_ppife_1_100_right}
\end{table}


\noindent
{\bf Orthocircle Interface}: In this example, the interface problem \eqref{model} has an interface surface with a more sophisticated geometry/topology than the
one in the previous example. In particular, the interface surface in this example is an orthocircle that is topologically isomorphic to the surface formed by three orthogonal torus as shown in Figure \ref{fig:torus3}. The interface can be described by a level-set:
$\{X\in \Omega : w(X)=0\}$ with $w(X) :=k(w_1(X)w_2(X)w_3(X)-\rho)$ in which
\begin{subequations}
\begin{align}
\label{f_123_func}
    &w_1(X) = w_1(x_1,x_2,x_3) = (x_1^2 + x_2^2 + x_3^2 + R^2 - r^2)^2 - 4R^2(x_2^2 + x_3^2),  \\
    &w_2(X) = w_2(x_1,x_2,x_3) = (x_1^2 + x_2^2 + x_3^2 + R^2 - r^2)^2 - 4R^2(x_1^2 + x_3^2),   \\
    &w_3(X) = w_3(x_1,x_2,x_3) = (x_1^2 + x_2^2 + x_3^2 + R^2 - r^2)^2 - 4R^2(x_1^2 + x_2^2),
\end{align}
\end{subequations}
and $\rho=0.15$, $k=10^{-3}, r=0.3, R=1$. We choose $f$ and $g$ in the interface problem \eqref{model} so that its exact solution is $u=u^s(x_1,x_2,x_3)=w(x_1,x_2,x_3)/\beta^s$, $s=\pm$ with $\beta^-=1$, $\beta^+ =100$. The Lagrange interpolation errors and PPIFE solution errors as well as the related convergence rates are presented in Tables \ref{table:examp_torus3_interp_1_100} and \ref{table:examp_torus3_ppife_1_100}. The data in these tables clearly demonstrate the optimal convergence, and this example indicates that the proposed IFE method can handle interface problems with quite complicated geometries and topologies.

\begin{table}[H]
\begin{center}
\scriptsize{
\begin{tabular}{|c |c c|c c|c c|}
\hline
$h$   & $\| u - I_hu\|_{L^{\infty}(\Omega)}$ & rate  & $\| u - I_hu\|_{L^2(\Omega)}$ & rate   & $| u - I_hu|_{ PH^1(\Omega)}$ & rate   \\ \hline
1/30   &    2.0451e-02          &                   &   7.1911e-03          &                   &   1.4047e-01   &                   \\ \hline
1/40   &    1.2089e-02          &    1.8274    &   4.0479e-03          &   1.9975      &   1.0539e-01   &   0.9988    \\ \hline
1/50  &     7.9730e-03          &    1.8655    &   2.5916e-03          &   1.9984      &   8.4327e-02   &   0.9993     \\ \hline
1/60  &     5.6492e-03          &    1.8897    &   1.8001e-03          &   1.9989      &   7.0279e-02   &   0.9995     \\ \hline
1/70  &     4.2107e-03          &    1.9066    &   1.3227e-03          &   1.9992      &   6.0243e-02   &   0.9996     \\ \hline
1/80  &     3.2589e-03          &    1.9189    &   1.0128e-03          &   1.9994      &   5.2715e-02   &   0.9997     \\ \hline
1/90  &     2.5967e-03          &    1.9284    &   8.0025e-04          &   1.9995      &   4.6859e-02   &   0.9998     \\ \hline
1/100  &   2.1176e-03          &    1.9358   &   6.4823e-04          &    1.9996      &   4.2174e-02   &    0.9998      \\ \hline
\end{tabular}
}
\end{center}
\caption{Errors and rates of the Lagrange interpolation for $\beta^-=1$, $\beta^+=100$.}
\label{table:examp_torus3_interp_1_100}
\end{table}

\begin{table}[H]
\begin{center}
\scriptsize{
\begin{tabular}{|c |c c|c c|c c|}
\hline
$h$   & $\| u - u_h\|_{L^{\infty}(\Omega)}$ & rate  & $\| u - u_h\|_{L^2(\Omega)}$ & rate   & $| u - u_h|_{ PH^1(\Omega)}$ & rate   \\ \hline
1/30   &    2.0458e-02          &                  &    8.1424e-03          &                    &    1.4051e-01   &                     \\ \hline
1/40   &    1.2092e-02          &    1.8278    &    4.6044e-03          &    1.9816     &    1.0552e-01   &   0.9955    \\ \hline
1/50  &     7.9741e-03          &    1.8658   &    2.9566e-03          &      1.9852   &     8.4453e-02   &   0.9979     \\ \hline
1/60  &     5.6498e-03          &    1.8899   &    2.0543e-03          &      1.9969   &     7.0377e-02   &   1.0000     \\ \hline
1/70  &     4.2110e-03           &    1.9067    &   1.5104e-03          &      1.9953    &    6.0326e-02   &   0.9997      \\ \hline
1/80  &     3.2591e-03          &    1.9190   &    1.1575e-03          &      1.9931   &     5.2785e-02   &   1.0001     \\ \hline
1/90  &     2.5969e-03          &    1.9285   &    9.1452e-04          &      2.0001   &     4.6917e-02   &   1.0004     \\ \hline
1/100  &   2.1177e-03           &    1.9359    &    7.4006e-04          &     2.0090    &   4.2220e-02    &   1.0013       \\ \hline
\end{tabular}
}
\end{center}
\caption{Errors and rates of the PPIFE solutions for $\beta^-=1$, $\beta^+=100$.}
\label{table:examp_torus3_ppife_1_100}
\end{table}


\noindent
{\bf Interface with Multiple Components}: We note that all the results developed in this article can be readily extended to treat interface problems whose interface surface separates the solution domain $\Omega$ into more than two subdomains.
For a demonstration, the interface surface in this example has multiple components while the interface surfaces in the previous two examples do not have. Specifically, the interface is the union of two disjoint surfaces, one surface is a sphere defined by the level-set $\{X \in \Omega : w_1(X) = 0\}$ where $w_1(X) = (x-x_0)^2 + (y-y_0)^2 + (z-z_0)^2 - \rho^2$ with $x_0=0.3$, $y_0=0$, $z_0=0$, $\rho=0.5$, the other surface is a torus defined by the level-set $\{X \in \Omega : w_2(X) = 0\}$ where $w_2(X) = (x_1^2 + x_2^2 + x_3^2 + R^2 - r^2)^2 - 4R^2(x_2^2 + x_3^2)$ with $R=1$, $r=0.3$.
These two interface components separate $\Omega$ into three subdomains $\Omega^i, i = 1, 2, 3$
such that the one inside the torus is $\Omega^1$, the inside of the sphere is $\Omega^2$, and
$\Omega^3 = \Omega - (\Omega^1 \cup \Omega^2)$. We choose $f$ and $g$ in the interface problem \eqref{model} so that
its exact solution is $u^s = f_1f_2/\beta^s$ on $\Omega^s$, $s=1,2,3$. The related numerical results reported in Tables \ref{table:examp_torus3_interp_1_100} and \ref{table:examp_torus3_ppife_1_100} clearly show the optimal convergence, and this demonstrates that the proposed IFE method can handle the interface problems whose interface consists of disjoint surfaces.

\begin{table}[H]
\begin{center}
\scriptsize{
\begin{tabular}{|c |c c|c c|c c|}
\hline
$h$   & $\| u - I_hu\|_{L^{\infty}(\Omega)}$ & rate  & $\| u - I_hu\|_{L^2(\Omega)}$ & rate   & $| u - I_hu|_{ PH^1(\Omega)}$ & rate   \\ \hline
1/30   &    5.9738e-02          &                   &   3.2079e-02           &                   &    7.2930e-01  &                   \\ \hline
1/40   &    3.6598e-02          &    1.7032    &   1.8643e-02           &   1.8866       &   5.5571e-01  &   0.9450    \\ \hline
1/50  &     2.4701e-02          &    1.7618    &   1.2138e-02           &   1.9231       &   4.4942e-01  &   0.9513     \\ \hline
1/60  &     1.7373e-02          &    1.9301    &   8.5212e-03           &   1.9404       &   3.7737e-01  &   0.9583     \\ \hline
1/70  &     1.2844e-02          &    1.9597    &   6.3182e-03           &   1.9405       &   3.2538e-01  &   0.9617     \\ \hline
1/80  &     1.0528e-02          &    1.4886    &   4.8800e-03           &   1.9343       &   2.8601e-01  &   0.9658     \\ \hline
1/90  &     8.4285e-03          &    1.8886    &   3.8756e-03           &   1.9566       &   2.5511e-01  &   0.9707     \\ \hline
1/100  &   6.8397e-03          &    1.9825    &   3.1525e-03           &   1.9598       &  2.3010e-01  &    0.9794      \\ \hline
\end{tabular}
}
\end{center}
\caption{Errors and rates of the Lagrange interpolation for $\beta^-=1$, $\beta^+=100$.}
\label{table:examp_torus_cic_interp_1_100}
\end{table}

\begin{table}[H]
\begin{center}
\scriptsize{
\begin{tabular}{|c |c c|c c|c c|}
\hline
$h$   & $\| u - u_h\|_{L^{\infty}(\Omega)}$ & rate  & $\| u - u_h\|_{L^2(\Omega)}$ & rate   & $| u - u_h|_{ PH^1(\Omega)}$ & rate   \\ \hline
1/30   &    8.6433e-02           &                  &    5.4842e-02         &                    &    8.2907e-01   &                     \\ \hline
1/40   &    5.0246e-02           &    1.8855    &   2.9779e-02          &     2.1226      &   6.0394e-01   &  1.1013   \\ \hline
1/50  &     3.4425e-02           &    1.6947   &    1.8322e-02         &      2.1766    &    4.7771e-01   &   1.0508    \\ \hline
1/60  &     2.3978e-02           &    1.9836   &    1.2392e-02         &      2.1451    &    3.9665e-01   &   1.0199    \\ \hline
1/70  &     1.7555e-02           &    2.0228    &   8.7675e-03         &      2.2443     &   3.3864e-01   &   1.0257     \\ \hline
1/80  &     1.3198e-02           &    2.1363   &    6.6078e-03         &      2.1179    &    2.9577e-01   &   1.0136    \\ \hline
1/90  &     1.0556e-02           &    1.8961   &    5.1240e-03         &      2.1592    &    2.6214e-01   &   1.0248    \\ \hline
1/100  &   8.6005e-03           &    1.9448    &   4.0734e-03          &     2.1777     &   2.3567e-01   &   1.0102      \\ \hline
\end{tabular}
}
\end{center}
\caption{Errors and rates of the PPIFE solutions for $\beta^-=1$, $\beta^+=100$.}
\label{table:examp_torus_cic_ppife_1_100}
\end{table}


\begin{appendices}
  \section{Technical Results}
  In this section, we present the proof of the two technical results.

  \subsection{Proof of Lemma \ref{lemma_gammadelta_01}}
  \label{App_A2}
  The proof is based on direct calculation. Without loss of generality, we only need to consider the interface element configuration shown in Figure \ref{fig:subfig_interf} with the vertices:
\begin{equation}
\begin{split}
\label{vertice}
A_1=(0,0,0), ~ A_2=(h,0,0), ~ A_3=(0,h,0), ~ A_4=(h,h,0),\\
  A_5=(0,0,h), ~ A_6=(h,0,h), ~ A_7=(0,h,h), ~ A_8=(h,h,h).
\end{split}
\end{equation}
and let the subelement containing $A_1$ be $T^-$. Here, we show a detailed discussion for the Case 1 in Figure \ref{fig:subfig_interf}(\subref{inter_elem_case1}).
Let $D_1=(d_1,0,0)$, $D_2=(0,d_2,0)$ and $D_3=(0,0,d_3)$ with $d_i\in[0,1]$, i=1,2,3. Then, we have
\begin{equation}
\begin{split}
\label{lgammadelta_eq_1}
\bfgamma^T\bfdelta = &\frac{1}{9 (d_2^2 d_3^2 +
   d_1^2 (d_2^2 + d_3^2))} ( d_1 d_2 d_3 ((-3 + d_2) d_2 (-3 + d_3) d_3 \\
   & - 3 d_1 (-3 d_2 + d_2^2 + (-3 + d_3) d_3) +
   d_1^2 (-3 d_2 + d_2^2 + (-3 + d_3) d_3)) ) \in [0,\frac{4}{9}].
   \end{split}
\end{equation}
In addition, note that $\|\bfdelta\|_{\infty}=|L(A_1)|$ for this case, then we can verify
\begin{equation}
\begin{split}
\label{lgammadelta_eq_2}
\frac{h^{-1} L(A_1)}{\bfgamma^T\bfdelta}
= &- 9 (d_2^2 d_3^2 + d_1^2 d_2^2 + d_1^2 d_3^2)^{1/2}  [ (-3 + d_2) d_2 (-3 + d_3) d_3 -
     3 d_1 (-3 d_2 + d_2^2 + (-3 + d_3) d_3)\\ +
     & d_1^2 (-3 d_2 + d_2^2 + (-3 + d_3) d_3) ]^{-1} \in [-2.25,-0.5776].
\end{split}
\end{equation}
The results for other cases can be proved by a similar argument, and here, we only list the bounds directly:
\begin{itemize}[leftmargin=45pt]
  \item[Case 2.] $\bfgamma^T\bfdelta \in [0,0.52]$ and $h^{-1}L(A_1)/\bfgamma^T\bfdelta\in[-3,0]$, $h^{-1}L(A_5)/\bfgamma^T\bfdelta\in[-2.133,0]$.
  \item[Case 3.]  $\bfgamma^T\bfdelta \in [0,1]$ and $h^{-1}L(A_1)/\bfgamma^T\bfdelta\in[0,1.664]$, $h^{-1}L(A_3)/\bfgamma^T\bfdelta\in[-1.039,2.603]$,\\
   $h^{-1}L(A_5)/\bfgamma^T\bfdelta\in[0,1.482]$, $h^{-1}L(A_7)/\bfgamma^T\bfdelta\in[0,1.664]$.
  \item[Case 4.]  $\bfgamma^T\bfdelta \in [2/9,8/9]$ and $h^{-1}L(A_1)/\bfgamma^T\bfdelta\in[-1.5,0]$, $h^{-1}L(A_2)/\bfgamma^T\bfdelta\in[-2.607,0]$,\\
   $h^{-1}L(A_3)/\bfgamma^T\bfdelta\in[-2.607,0]$.
   \item[Case 5.]  $\bfgamma^T\bfdelta \in [1/6,0.52]$ and $h^{-1}L(A_1)/\bfgamma^T\bfdelta\in[-7.43,-1.415]$, $h^{-1}L(A_2)/\bfgamma^T\bfdelta\in[-3.222,0]$,\\
    $h^{-1}L(A_3)/\bfgamma^T\bfdelta\in[-3.222,0]$, $h^{-1}L(A_5)/\bfgamma^T\bfdelta\in[-4.867,0]$.
\end{itemize}


  \subsection{Proof of Lemma \ref{lem_upm_tau}}
  \label{App_A1}
For each $T\in\mathcal{T}^i_h$, consider a local Cartesian system $\xi$, $\eta$, $\zeta$ where $\xi$, $\eta$ span the plane $\tau(T)$ and $\zeta$ is perpendicular to $\tau(T)$. Let $\zeta=f(\xi,\eta)$ be the equation of the surface $\Gamma_{\omega_T}$. Without causing any confusing, we still use $X$ to denote points in this local system. For each point $X=(\xi,\eta,0)\in\tau_{\omega_T}$, let $\widetilde{X}=(\xi,\eta,\tilde{\zeta})$ be the corresponding point on $\Gamma_{\omega_T}$ with $|X-\widetilde{X}| = |\tilde{\zeta}|\leqslant C  h^2$ by \eqref{interf_element_est_eq0_dist}. Firstly, for \eqref{upm_tau_eq01}, we let $v=u^+_E - u^-_E \in H^2(\omega_T)$, then $v|_{\Gamma}=0$. The Taylor expansion for $v(X)$ around $\widetilde{X}$ yields
\begin{equation}
\label{upm_tau_eq-1}
0=v(\widetilde{X}) = v(X) + \nabla v(X)\cdot(\widetilde{X}-X) + \int_{0}^{\tilde{\zeta}} \frac{\zeta}{2} \partial^2_{\zeta}v d\zeta
\end{equation}
Then, the triangular inequality yields
\begin{equation}
\label{upm_tau_eq-2}
\left( \int_{\tau_{\omega_T}} v^2 dX \right)^{1/2} \leqslant |X-\widetilde{X}| \left( \int_{\tau_{\omega_T}} (\nabla v)^2 dX \right)^{1/2} +  \left( \int_{\tau_{\omega_T}} \left( \int_{0}^{\tilde{\zeta}} \frac{\zeta}{2} \partial^2_{\zeta}v d\zeta \right)^2 dX  \right)^{1/2}
\end{equation}
For the first term in \eqref{upm_tau_eq-2}, we apply the trace inequality \eqref{trace_inequa_pyd_eq0} to obtain
\begin{equation}
\label{upm_tau_eq-3}
|X-\widetilde{X}| \left( \int_{\tau_{\omega_T}} (\nabla v)^2 dX \right)^{1/2} \leqslant C  h^2 \| \nabla v \|_{L^2(\tau_{\omega_T})}
\leqslant C  (  h^{3/2} | v |^2_{H^1(\omega_T)} +  h^{5/2} | v |^2_{H^2(\omega_T)} ).
\end{equation}
For the second term in \eqref{upm_tau_eq-2}, we apply the H\"older's inequality to obtain
\begin{equation}
\label{upm_tau_eq-4}
\left( \int_{\tau_{\omega_T}} \left( \int_{0}^{\tilde{\zeta}} \frac{\zeta}{2} \partial^2_{\zeta}v d\zeta \right)^2 dX \right)^{1/2}  \leqslant  |\tilde{\zeta}|^{3/2} \left( \int_{\tau_{\omega_T}} \int_{0}^{\tilde{\zeta}} (\partial^2_{\zeta}v)^2 d\zeta dX \right)^{1/2} \leqslant C h^3 |v|_{H^2(\omega_T)}.
\end{equation}
Putting \eqref{upm_tau_eq-3} and \eqref{upm_tau_eq-4} into \eqref{upm_tau_eq-2}, we have \eqref{upm_tau_eq01}.

For \eqref{upm_tau_eq02}, we let $w=\beta^+ u^+_E - \beta^-  u^-_E$, then $\nabla w\cdot\mathbf{ n}=0$ on $\Gamma$ where $\mathbf{ n}$ is the normal vector to $\Gamma$. Firstly, we use \eqref{interf_element_est_diff} and trace inequality  given by Lemma 3.2 in \cite{2016WangXiaoXu} to obtain
\begin{equation}
\begin{split}
\label{upm_tau_eq1}
\| \nabla w\cdot \bar{\mathbf{ n}} \|_{L^2(\Gamma_{\omega_T})} & = \| \nabla w\cdot ( \bar{\mathbf{ n}} - \mathbf{ n } ) \|_{L^2(\Gamma_{\omega_T})} \leqslant C  h \| \nabla w \|_{L^2(\Gamma_{\omega_T})} \\
& \leqslant  C (  h^{1/2}\| w \|_{H^1(\omega_T)}  +   h^{3/2} \| w \|_{H^2(\omega_T)} ).
\end{split}
\end{equation}
Note that $\zeta$ is in the direction $\bar{\mathbf{ n}}$. So by applying the first order Taylor expansion to $\nabla w\cdot\bar{\mathbf{ n}}$ and using similar argument to \eqref{upm_tau_eq-2}, we have
\begin{equation}
\begin{split}
\label{upm_tau_eq2}
\left( \int_{\tau_{\omega_T}} ( \nabla w(X)\cdot\bar{\mathbf{ n}} )^2 d\xi d\eta \right)^{1/2} 
\leqslant \left( \int_{\tau_{\omega_T}} (\nabla w(\widetilde{X})\cdot\bar{\mathbf{ n}})^2 d\xi d\eta \right)^{1/2} +  \left(  \int_{\tau_{\omega_T}} \left( \int_{0}^{\tilde{\zeta}} \partial^2_{\zeta} w \, d \zeta  \right)^2 dX \right)^{1/2}.
\end{split}
 \end{equation}
 From \eqref{upm_tau_eq1}, \eqref{interf_element_est_prod} and trace inequality given by Lemma 3.2 in \cite{2016WangXiaoXu}, we have
 \begin{equation}
 \begin{split}
\label{upm_tau_eq3}
\left( \int_{\tau_{\omega_T}} (\nabla w(\widetilde{X})\cdot\bar{\mathbf{ n}})^2 d\xi d\eta \right)^{1/2} &= \left( \int_{\Gamma_{\omega_T}} (\nabla w(\widetilde{X})\cdot\bar{\mathbf{ n}})^2 \frac{1}{\bar{\mathbf{ n}}\cdot\mathbf{ n}(\widetilde{X}) } dS \right)^{1/2} \leqslant C \| \nabla w(\widetilde{X})\cdot\bar{\mathbf{ n}} \|^2_{L^2(\Gamma_{\omega_T})} \\
& \leqslant  C (  h^{1/2}\| w \|_{H^1(\omega_T)}  +   h^{3/2} \| w \|_{H^2(\omega_T)} ).
\end{split}
\end{equation}
Next, by H\"older's inequality similar to \eqref{upm_tau_eq-4}, we have
\begin{equation}
\label{upm_tau_eq4}
\left( \int_{\tau_{\omega_T}} \left( \int_{0}^{\tilde{\zeta}} \partial^2_{\zeta} w \, d \zeta  \right)^2 d\xi d\eta \right)^{1/2}
\leqslant |\tilde{\zeta}|^{1/2} \left( \int_{\tau_{\omega_T}}  \int_{0}^{\tilde{\zeta}} ( \partial^2_{\zeta} w )^2 \, d \zeta  dX \right)^{1/2} \leqslant C h | w|_{H^2(\omega_T)}.
\end{equation}
Finally, substituting \eqref{upm_tau_eq3} and \eqref{upm_tau_eq4} to \eqref{upm_tau_eq2}, we arrive at \eqref{upm_tau_eq02}.

\end{appendices}


\end{document}